\documentclass[12pt,oneside,reqno]{amsart}
\usepackage{graphicx}
\usepackage[breaklinks=true,colorlinks=true,linkcolor=blue,urlcolor=blue,citecolor=blue]{hyperref}

\usepackage[letterpaper, portrait, margin=1in]{geometry}
\usepackage{amssymb}
\newcommand{\R}{{\mathbb R}}
\newcommand{\Z}{{\mathbb Z}}

\newtheorem{thm}{Theorem}[section]
\newtheorem{coro}[thm]{Corollary}
\newtheorem{definition}[thm]{Definition}
\newtheorem{lem}[thm]{Lemma}
\newtheorem{pro}[thm]{Proposition}
\newtheorem{rem}[thm]{Remark}

\def\ii{\'\i}

\begin{document}

\title[]{Stability of non-monotone and backward waves for delay non-local reaction-diffusion equations}

\author[Solar]{Abraham Solar}
\address{Instituto de F\ii sica,  Facultad de F\ii sica
P. Universidad Cat—\'olica de Chile 
Casilla 306, Santiago 22, Chile }
\email{asolar@fis.puc.cl}

\maketitle

\begin{abstract} 
This paper deals with the stability of semi-wavefronts to the following
delay non-local monostable  equation: $\dot{v}(t,x) = \Delta v(t,x) - v(t,x) + \int_{\R^d}K(y)g(v(t-h,x-y))dy, x \in \R^d,\ t >0;$ where $h>0$ and $d\in\Z_+$. We give two general results for $d\geq1$: on the global stability of semi-wavefronts in $L^p$-spaces with unbounded weights  and the local stability of planar wavefronts in $L^p$-spaces with bounded weights. We also give a global stability result for $d=1$ which includes the global stability on Sobolev spaces.  Here $g$ is not assumed to be monotone and  the kernel $K$ is not assumed to be symmetric, therefore non-monotone semi-wavefronts  and  {\it backward traveling fronts} appear for which we show their stability. In particular, the global stability of critical wavefronts is stated.     
 \end{abstract}


\vspace{2pc}
\noindent{\it Keywords}: semi-wavefront, stability, non-local, delay, reaction-diffusion equations
\newpage


\section{Introduction}

We study the following non-local equation with delay
\begin{eqnarray}\label{nle0}
\dot{v}(t,x)=\Delta v(t,x)-v(t,x)+\int_{\R^d}K(y)g(v(t-h,x-y))dy\quad x\in\R^d, \ t>0,
\end{eqnarray}
   
\vspace{3mm}

\noindent which is an important  model in population dynamics \cite{SWZ, ThZh,WLR,TAT,MOZ, YCW, GPT,YZ,BS}, where the parameter $h>0$ is the sexual mature period of some species with birth rate $g$ (with equilibria $0$ and $\kappa>0$) and the non-local interaction between individuals is determined by the kernel $K$ while the quantity $v(t,x)$ stands for the mature population at some time $t$ and point $x$. In this context,  a kind of  colonization waves with constant propagation speed appears which are called \textit{planar semi-wavefronts}, i.e., solutions $v(t,x)=\phi_c(\nu\cdot x+ct)$ with speed $c\in\R$, $\nu\in\mathbb{S}^{d-1}$ and the profile $\phi_c:\R\to\R_+$ satisfying $\phi_c(-\infty)=0$ (or $\phi_c(+\infty)=0$) and  $\liminf_{z\to+\infty}\phi_c(z)>0$ (or $\liminf_{z\to-\infty}\phi_c(z)>0$); if $\phi_c(+\infty)=\kappa$ (or $\phi_c(-\infty)=\kappa$) then the semi-wavefronts are called {\it planar wavefronts}. Due to a possible asymmetry of $K$, the class of profiles satisfying $\phi_c(-\infty)=0$ could be different to the class of profiles satisfying $\phi_c(+\infty)=0$, therefore we must expect two minimal speeds for the existence of semi-wavefronts which could be non-opposite   \cite{W,LZh,GPT,YZ}.     

In the non-delayed local case semi-wavefronts are monotone wavefronts and the study of existence, uniqueness, asymptotic spreading speeds  and stability  is widely documented \cite{AW,ES,Go,HR,K,BS,STG,ST,UC}. Broadly speaking, it has been shown that the asymptotic propagation speed of solutions  only depends on the asymptotic behavior of  initial datum at the trivial equilibrium, i.e., two initial data could coincide on some domain $(N,+\infty]$, for arbitrary $N\in\R$, but if their asymptotic behavior at $-\infty$ are different then they will be propagated with different speeds. Particularly, Kolmogorov, Petrovskii and Piskunov \cite{KPP} showed that if the initial datum is the Heaviside step function the solution is propagated with the critical speed. Due to this result, the critical wavefronts have been one of the main edges in the research on this subject. The model in \cite{KPP}  satisfies the \textit{subtangential} property $g(u)\leq g'(0)u,$ for all $u\geq0$, which implies that the critical wavefronts are propagated with the linear speed $c_*=2\sqrt{1-g'(0)}$, those minimal wavefronts  we will consider in this paper.   

Many delay models were presented after \cite{KPP} and the research was addressed to similar problems \cite{G,LLLM,MW,MOZ}. One of the most cited models is the Nicholson's blowflies model that, in its non-local diffusive version, is 
\begin{eqnarray} \label{nich}      
\dot{v}(t,x)=\Delta v(t,x)-\delta v(t,x)+p\int_{\R^d}K(x-y)v(t-h,y)e^{-v(t-h,y)}dy,
\end{eqnarray}
for some positive parameters $\delta$ and $p$, which is reduced to (\ref{nle}) by an appropriate rescaling of variables. 

Since the nonlinearity in (\ref{nich}) satisfies $|g|_{Lip}=g'(0)$ (here we use the $|g|_{Lip}$ to denote the Lipschitz constant of $g$), the uniqueness (up to translation) of semi-wavefronts to (\ref{nich}) is a consequence from \cite[Theorem 7]{AG}. Alternatively, we give a result on the uniqueness of non-critical semi-wavefronts of (\ref{nle0})(see Corollary \ref{cor}). Otherwise,  the existence of wavefronts (monotone and non-monotone) for (\ref{nle0}) has been studied, e.g., in \cite{W,LZh,ThZh,TAT,YZ} and results for the existence of semi-wavefronts has been given in \cite[Theorem 4]{TAT} and \cite[Theorem 18]{GPT} (by providing the existence of a minimal speed when $|g|_{Lip}=g'(0)$)  where  the kernel $K$ is not assumed to be even. A more complete discussion on the existence of wavefronts and semi-wavefronts is given in Subsection 2.2.  


In general, the study of delayed case mainly presents two troubles. The first one is concerned with the asymptotic behavior of semi-wavefronts in the positive equilibrium $\kappa$ since 
the associated characteristic equations have infinity solutions and semi-wavefronts could oscillate around $\kappa$. Indeed, non-monotone wavefronts to (\ref{nle0}) have been observed \cite{TAT,WLL,GPT,YZ}. 
Otherwise, the second trouble is  that the associated semi-flow  to (\ref{nle0}) is not  monotone in general. This lack complicates  the construction of sub and super-solution, an approach widely used when $h=0$ or $g$ is monotone to prove the existence and  stability of wavefronts.   The spectral technique has been used in order to obtain the local stability \cite{G,K,STG, sch}, however, the maximum principle arguments to reaction-diffusion equations frequently imply the global stability of wavefronts \cite{AW, UC,MOZ, STR1, HMW}. Nevertheless, our approach is a combination of maximum principle arguments and Fourier analysis for linear delay PDE's.

  For local equations (with $d=1$), when $g$ is sub tangential and possibly non-monotone  the local exponential stability of wavefronts, in  suitable Sobolev spaces, was given by  Lin $et\ al$ \cite{LLLM} (for non-critical wavefronts) and Chern {\it et al} \cite{CMYZ} (for critical wavefronts) under the condition $|g'(\kappa)|<1$ for any delay or $g'(\kappa)<-1$ for small delay.
Then, in \cite{LLLM}  the algebraic stability of semi-wavefronts with speed $c\geq c(|g|_{Lip})$, some speed $c(|g|_{Lip})$ (see Definition \ref{def} below), on any domain of the form $(-\infty,N],$ $N\in\R$,  was proved in \cite[Theorem 3]{S} without assumptions on neither subtangetiality of $g$ nor size of derivate on equilibrium $\kappa.$ In particular, when $|g|_{Lip}=g'(0)$ semi-wavefronts (including the critical and asymptotically periodic semi-wavefronts) are stable on any domain $(-\infty,N]$. This limitation on the stability domain is by the use of an unbounded weight so that the control of the stability of semi-wavefronts on its whole domain yields to the stability with a bounded weight. For the local stability in \cite[Corollary 17]{S} was showed that the size of local perturbations depends on the size of neighborhoods of $\kappa$ where $g$ is contractive application and one of these neighborhoods of $\kappa$ is attractor and therefore the global exponential stability of non-critical semi-wavefronts was also established in \cite[Corollary 11]{S}, which  includes non-monotone wavefronts for typical models such as local Nicholson's blowflies model (when $p/\delta\in[1,e^2]$) and Mackey-Glass' model. However, the stability of critical semi-wavefronts was not addressed in \cite{S} so that  we study the global stability of critical wavefronts  in this paper. 


In respect to the non-local equations, when $d=1$ the stability of wavefronts  has also been studied for bistable nonlinearity without delay (see, e.g., \cite{CH}) and with delay (see,e.g.,  \cite{WLRO} and \cite{MW}). In the monostable case (\ref{nle0}) with delay,  the global stability of the monotone wavefronts with monotone $g$ was satisfactorily answered by Mei {\it et al} in \cite{MOZ} when $K$ is a heat kernel. Similar results, for more general equation, were established by Lv and Wang \cite{LvW}. Also, a close model to (\ref{nle0}) is a paper of  Wang {\it et al} \cite{WLR} where the authors proved the global stability of non-critical (under minimal conditions on the initial data) when $g$ is monotone and $K$ is an even kernel. For $d\geq 1$, we should mention a very interesting work for dispersal equations presented by Huang {\it et al} \cite{HMW} where the global stability of monotone  planar wavefronts was stated and the study of the convergence rate was dealt; here $K$ is a multidimensional heat kernel.   So that, as much as we know the study of stability of semi-wavefronts for the non-local case assumes the monotonicity of $g$ and the symmetry  of $K$.  Thus, our aim is to prove  the global stability of wavefronts which could be backward wavefronts or oscillatory wavefronts. In particular, our global stability result for  asymmetric kernel implies a change of behavior in the problem of {\it speeds selection} for the equation (\ref{nle0}), i.e., to determinate the asymptotic speed propagation of solutions generated by an initial data  by only knowing the asymptotic behavior of the initial data at the trivial equilibrium (see Remark \ref{sspeed}).

In respect to the convergence rate of solutions to critical semi-wavefronts our result of local stability is comparable to Gallay work \cite{G} for local equations without delay. More precisely, the disturbances space in \cite{G} is  a subspace of our disturbance space in the sense that the weights defer by a quadratic factor. Although, in our space the convergence is $O(t^{-1/2})$ while in the subspace considered in \cite{G} is faster than $O(t^{-3/2}).$ Also, the convergence rate in our global stability result extends the pioneering result of Mei {\it et al} \cite{MOZ} in Sobolev spaces for (\ref{nle0}) (see Corollary \ref{sobolev} below) without requiring the convergence of the initial datum to $\kappa.$   This paper is matched with a recent work of Benguria and Solar \cite{BS} where it is showed that the algebraic convergence rate for critical wavefronts obtained in this paper is optimal in the underline weigthed space and it also has a closed relation with convergence rate obtained in \cite{HMW}. 


We organize this paper in the following way.  In the Section 2 we present and discuss the main results, in the Section 3 we state an existence and regularity result for the Cauchy problem, in the Section 4 we prove the stability results for $d\geq 1$ (stability on semi-intervals and local stability) and finally, in Section 5 we prove the  global stability result for $d=1$. 
  
\section{Main Results and Discussion}
\vspace{3mm}
\subsection{Global stability with exponential weight on $\R^d$}
\vspace{3mm}


\vspace{2mm}

\noindent Now, in order to study the stability of semi-wavefronts with speed $c$ in the direction $\nu\in \mathbb{S}^{d-1}$ we make the change of variables $z:=x+ct\nu$ and $u(t,z):=v(t,z-ct\nu)$, so that we have the following equation for $u$
\begin{eqnarray}\label{nle}
\dot{u}(t,z)=\Delta u(t,z)-c\nu\cdot\nabla u(t,z)-u(t,z)+K\ast [g\circ u](t-h,z-ch\nu),\ z\in\R^d,
\end{eqnarray}
 for which the planar semi-wavefronts $v(t,x)=\phi_c(\nu\cdot x+ct)$ with speed $c$, $\phi_c:\R\to\R_+$, are stationary solutions $u(t,z)=\phi_c(\nu\cdot z)$  the following equation 
 
  \begin{eqnarray}\label{este}
\phi_c''(\nu\cdot z)-c\phi_c'(\nu\cdot z)-\phi_c(\nu\cdot z)+\int_{\R^d}K(z-ch\nu -y)g(\phi_c(\nu\cdot y))dy=0\quad \hbox{for all} \ z\in\R^d,
\end{eqnarray}

In our first stability result  we do not assume neither differentiability nor subtangentiality on $g$. Our general assumption on $g$ is the following   

\bigskip
 {\rm \bf(L)}  The function $g:\R\rightarrow\R $  is Lipschitz continuos with constant $|g|_{Lip}.$

\bigskip
 
\noindent By denoting  $\xi_{\lambda}(z):=e^{-\lambda\cdot z}$ we have the following  linear equation  associated with (\ref{nle})  

\small\begin{eqnarray}\label{lde}
\dot{r}(t,z)=\Delta r(t,z)+(2\lambda-c\nu)\cdot\nabla r(t,z)+p_{\lambda}r(t,z)+|g|_{Lip}\ e^{-\lambda\cdot\nu ch}[\xi_{\lambda} K\ast r](t-h,z-ch\nu),
\end{eqnarray}\normalsize

\noindent where $p_{\lambda}=p_{\lambda}(c)=|\lambda|^{2}-c\nu\cdot\lambda-1$. Also, we denote by

$$
q_{\lambda}=q_{\lambda}(c)=|g|_{Lip} \ e^{-\lambda\cdot\nu ch}\int_{\R^d}K(y)e^{-\lambda\cdot y}dy  
$$

The behavior of solutions to (\ref{lde}), in the state space $L^1(\R)$ with certain exponential weight has been studied, e.g., in \cite{BS}.
Otherwise, for $r\in\Z_+\cup\{0\}$ and $1\leq p\leq \infty$ we denote the weighted  Sobolev spaces
$$
W_{\lambda}^{r,p}:=\{u:\R^{d}\to\R, \ \hbox{such that}\  |u|_{W_{\lambda}^{r,p}}:=||\xi_{\lambda}u||_{W^{r,p}(\R^d)}<\infty \}
$$  
  and 
 \small $$
W_{h,\lambda}^{r,p}:=C([-h,0], W^{r,p}_{\lambda})=\{u:[-h,0]\to W^{r,p}_{\lambda} \ \hbox{continuous with norm}\ |u|_{W^{r,p}_{h,\lambda}}=\max_{s\in[-h,0]}|u(s)|_{L_{\lambda}^{r,p}} \}  
  $$\normalsize
When $r=0$ the letter $W$ is replaced by $L$. Analogously, we define weighted  Holder spaces $C^{r,p}_{\lambda}$ and $C^{r,p}_{h,\lambda}$.  

Finally, for some function $u: [a-h,b]\to X$, some Banach space $X$ and $a,b\in\R$, we  define $u_t:[-h,0]\to X$ for each $t\in[a,b]$ as $u_t(s)= u(t+s)$.

\vspace{3mm}
Now, we state our first result on the stability of solutions to (\ref{nle}).

\begin{thm}\label{st}

Assume {\bf (L)} and fix $c\in\R$ and  $\lambda, \lambda'\in\R^d$ such that $K\in L^1_{\lambda}\cap L^1_{\lambda'}$. 
If $u_0,\psi_0\in L_{h, \lambda}^{\infty}\cap C_{h, \lambda'}^{0,\alpha}$, some $\alpha\in(0,1)$, and
 \begin{eqnarray}\label{ic}
 r_0(s,z):=|u_0(s,z)-\psi_0(s,z)|\in L^1_{h,\lambda}
 \end{eqnarray}
 then $u_t(\cdot,\cdot)$ and $\psi_t(\cdot,\cdot)$ uniquely exist in $L^{\infty}_{h, \lambda}\cap C^{0, \alpha}_{h, \lambda}$ for all $t\geq -h$. Moreover, if  $r(t, z)$  satisfies (\ref{lde})  with the initial datum $\xi_{\lambda}r_0$ we obtain 
\begin{eqnarray}\label{rc1}
|u(t,z)-\psi(t,z)|\leq r(t,z)\leq A_{\lambda} |r_0|_{L^1_{h,\lambda}}\ \frac{e^{-\gamma_{\lambda} t}}{t^{d/2}}e^{\lambda\cdot z}\ \hbox{for all}\ t> h(d+1)/2\ \hbox{and} \ z\in\R^d,
\end{eqnarray}
where 
$$
A_{\lambda}:=   (\frac{ 1+hq_{\lambda}e^{\gamma_{\lambda}h}}{4\pi})^{d/2} 
$$

\noindent and $\gamma_{\lambda}$ is the unique real solution of the following equation 
\begin{eqnarray}\label{P}
\gamma+p_{\lambda}+q_{\lambda}e^{h\gamma}=0,
\end{eqnarray}

\noindent In particular, if $\phi_c$ is a stationary solution of (\ref{nle}) and $q_{\lambda}\leq -p_{\lambda}$, i.e. $\gamma_{\lambda}\leq 0$, then $\phi_c$ is globally stable in $L^1_{h,\lambda}\cap L^{\infty}_{h,\lambda}$  with rate convergence $O(t^{-d/2}e^{-\gamma_{\lambda }t})$.
\end{thm} 

\begin{rem}
When the initial datum is taken in $C^{0,0}_{h, \lambda}=BUC(\R^d)$ it is possible to prove the existence of a  {\it mild solution} of (\ref{nle0}) on $(0, +\infty)\times\R$, see Remark \ref{classic}. Moreover,  it also is possible to prove that for $t>h$ that solution is a {\it classic solution} of (\ref{nle0}), see  Proposition \ref{reg} and compare with \cite[Theorem 3.3]{WLR}.  
\end{rem}

\begin{rem}
For each $y\in\R^d$ we write $y=(\tilde{y},y_d)$ with $\tilde{y}\in\R^{d-1}$. We fix $(\tilde{\nu},\nu_d)\in\R^{d}$ and $\tilde{\nu}\in\R^{d-1}$ and define
$$
E_c(\lambda_d):= \lambda_d^2-c\nu_d\lambda_d -1-c\tilde{\nu}\cdot\tilde{\lambda}+ q^*e^{-ch\lambda_d\nu_d}\int_{\R}e^{-\lambda_d y_d}\mathcal{K}(\lambda_d)dy_d,
$$
 where $\mathcal{K}(s):=\int_{\R^{d-1}}e^{-\tilde{\lambda}\cdot y}K(y_1,...,y_{d-1},s)dy\in L^1(\R) $ and $q^*:=|g|_{Lip}\ e^{-ch\tilde{\nu}\cdot\tilde{\lambda}} $. Next, if $$0<c\tilde{\nu}\cdot\tilde{\lambda}+1<q^*||\mathcal{K}||_{L^1(\R)}$$ we can invoke \cite[Lemma 22]{GPT} in order to obtain two numbers $c_*^-=c^-_*(\tilde{\nu}, \tilde{\lambda})$  and $c_*^+=c^+_*(\tilde{\nu}, \tilde{\lambda})$ such that if $c\geq c_*^+/\nu_d$ or $c\leq c_*^-/\nu_d$ then there exist at least a number $\lambda_d^*$ such that 
\begin{eqnarray}\label{er} 
q_{(\tilde{\lambda}, \lambda_d^*)} \ \leq \ -p_{(\tilde{\lambda},\lambda_d^*)}.
 \end{eqnarray}
 In particular, if   $\tilde{\lambda}=0$ and $|g|_{Lip}=g'(0)>1$ (monostable type) then (\ref{er}) is satisfied and therefore each solution of (\ref{este}) with $ c\geq c_*^+/\nu_d$ or $c\leq c_*^-/\nu_d$  is globally stable in $L^{p}_{h, ({\bf 0},\lambda_d^*)}$.
\end{rem}

We note that the Theorem \ref{st} shows that two semi-wavefronts are equal by a translation whenever their asymptotic terms of order one coincide, i.e., the condition (\ref{ic}). 
 
  \begin{coro}[Uniqueness of semi-wavefronts]\label{cor}
 Assume the condition {\bf (L)}. If $\phi_c$ and $\tilde{\phi}_c$ are stationary solutions of (\ref{nle})   such that $p_{\lambda'}(c)+q_{\lambda'}(c)\leq 0$ and  $\phi_c-\tilde{\phi}_c\in L^{1}_{\lambda'}$ for some $\lambda'\in\R$, then $\phi_c(\cdot+z_0)=\tilde{\phi}_c(\cdot)$
 for some $z_0\in\R.$
  \end{coro}

Naturally, because of (\ref{rc1}) the perturbation is maintained in the space $C([-h,0],L^1(\R)\cap L^{\infty}(\R))$ with weight $\xi_{\lambda }$ for all $t>h(d+1)/2.$ This fact is true for all $t\geq -h$ as it is showed in Proposition \ref{reg}. For $g\in L^{\infty, k}(\R)$  we give a result on the persistence of the derivates of perturbations in Section 3  which shows that for the derivates of order $k$ the persistence is obtained for $t>h(k-1)$. 

 This result generalizes \cite[Theorem 3]{S} which is referred to local equations. In the case $d=1$ and  $|g|_{Lip}=g'(0),$  equation (\ref{nle}) admits a semi-wavefront $\phi_c$ with speed $c\in\mathcal{C}$ (see,e.g., Proposition \ref{pro} below) and $\phi_c(z)= A_{\phi_c}z^{j_c}e^{\lambda_1(c) z}+O(e^{(\lambda_1(c)+\epsilon)z}),$ for some positive numbers $A_{\phi_c}$ and $\epsilon$,  and $j_c=0,1$ where $j_c=0$ if only if $c\in\mathcal{C}\setminus\{c_*^-,c_*^+\}$(see,e.g., \cite[Theorem 3]{AG}). Therefore, the semi-wavefront $\phi_c$  can be found, on any domain
$(-\infty,N], N\in\R,$ by means of the evolution of  any initial datum to (\ref{nle}) in the form $u_0(z)=A_{\phi_c}e^{\lambda_1(c) z}+ O(e^{(\lambda_1(c)+\epsilon)z})$  with bounds explicitly given in (\ref{rc1}) where the convergence rate is like $O(t^{-\frac{1}{2}}e^{-\gamma_{\lambda} t})$ for $\gamma_{\lambda}\geq 0$ determined by some $\lambda\in(\lambda_1(c),\lambda_1(c)+\epsilon) $ in (\ref{P}). For local equations, the  numerical simulations for the approximation to critical wavefronts done in \cite[Section 7]{CMYZ} can be used with the distance controlled by (\ref{rc1}). Also, in \cite{B} there are numerical simulations for monotone wavefronts to equation (\ref{nle0}).

\vspace{4mm}
We note that the Cauchy problem to (\ref{nle}) is well posed for non negative initial data since an application of maximum principle on unbounded domains (see, \cite[Theorem 10, Chapter 3]{PW})  implies that $u(t,\cdot)$ is positive for $t\in[0,h]$ and repeating this argument to intervals $[h,2h],[2h,3h],...$ we can conclude that $u(t,\cdot)$ is positive for all $t>0$.

\subsection{Existence of d-dimensional planar semi-wavefronts} 
 \vspace{3mm}
 The results on the existence of wavefronts of (\ref{este}) for monotone $g$ in an abstract setting are well known \cite{W,LZh,ThZh}. For non-monotone $g$, the existence of non monotone wavefronts to (\ref{este}) has been studied when $d=1$ \cite{HZ,TAT, WLL, GPT,YZ}. For $d>1$ and $K$ satisfying the following condition       
   \vspace{3mm}

 ({\rm\bf K}) There exist non-negative functions $K_i\in L^1(\R)$, for $i=1,...,d$, such that

  \begin{eqnarray*}
 K((s_1,...,s_d))=\displaystyle\prod_{i=1}^{d}K_i(s_i)\quad \hbox{and} \quad  \int_{\R^d}K(y)dy=1.
 \end{eqnarray*} 
 \indent Also, the function  $$\lambda\in\R\mapsto\int_{\R}K_i(y)e^{-\lambda y}dy,$$ 
 
 \indent is defined on some maximal open interval $(a_i,b_i)\ni 0,$\vspace{2mm}
\bigskip

\noindent the existence result for planar semi-wavefronts given in \cite[Theorem 18]{GPT} can be applied to (\ref{este}). For instance, in \cite{HMW} the authors  take $K_i$ equal to a heat kernel for all $i=1,...,d$.   

More precisely,  associated to equation (\ref{este}), for each $c\in\R$, we have the characteristic function $E^i_c:(a_i,b_i)\to\R$ defined by

\begin{eqnarray}\label{charc}
E^i_c(\lambda):=\lambda^2-c\lambda-1+g'(0)e^{-\lambda ch}\int_{\R}K_i(y)e^{-\lambda y}dy.
\end{eqnarray}



\noindent Without restriction of ({\bf K}) we can take $\int_{\R}K_i(s)ds=1$ for $i=1,...,d$. Next, if we fix a canonic vector {\bf e}, let us say ${\bf e}={\bf e}_1$, then $E_c:=E_c^1$ defined on the maximal open interval $(a_1, a_2)=:(a,b)\subset\R$ is the characteristic function associated to trivial equilibrium for wave's equation (\ref{este}) and therefore  according to \cite[Lemma 22]{GPT} we can make the following definition 

\begin{definition}\label{def}
 Denote by $c_*^-=c_*^-({\bf e}_1)<c_*^+=c_*^+({\bf e}_1)$ the two real numbers such that for each $c\in\mathcal{C}:=(-\infty, c_*^-]\cup[c_*^+,\infty)$ the function $E_c(\lambda)$ either (i) has exactly two real solutions $\lambda_1(c)\leq\lambda_2(c)$ or (ii) has exactly one real solution $\lambda_1(c)$. And if $c\in(c_*^-,c_*^+)$ then $E_c(\lambda)>0$ for all $\lambda\in(a,b)$.
\end{definition}

Also, if $c\leq c_*^-$ then the zeros of $E_c$ are negative while if $c\geq c_*^+$ then the zeros of $E_c$ are positive. Therefore, because of $E_c(0)>0$ and the continuity and monotony  of $E_c$ on the parameter $c$ the function $E_c$ is not positive in the compact interval defined by zeros of $E_c$.


\vspace{3mm}

Now, in order to establish the next results we make the following mono-stability condition.

\vspace{2mm}

{\rm \bf (M)}  The function $g:\R_{\geq 0}\rightarrow\R_{\geq 0} $ is bounded and  the equation $g(u)= u$ has 
\indent exactly two solutions: $0$ and $\kappa>0$. Moreover,  $g\in C^{1, \alpha}$ in some $\delta_0$-neighborhood 
\indent of zero  and $g$ is Lipschitz  with $|g|_{Lip}=g'(0) >1.$                    
\vspace{3mm}

\noindent Under conditions {\bf(M)} and {\bf(K)} the existence of semi-wavefronts was established, e.g., in \cite[Theorem 18]{GPT} and we present it as follow.

\begin{pro}\label{pro}
Suppose that $g$ satisfies {\bf(M)} and $K$ satisfies ${\bf(K)}$. Then for each $c\in\mathcal{C}$ the equation (\ref{nle0}) has a planar semi-wavefront  $v(t,x)=\phi_c(x\cdot{\bf e}_1+ct)$. Moreover, if $c\leq c_*^-$ then $\phi_c(+\infty)=0$ and if $c\geq c_*^+$ then $\phi_c(-\infty)=0.$ Also, if for some $\zeta_2=\sup_{s\geq 0}g(s)$ the equilibrium $\kappa$ is a global attractor of the map $g:(0,\zeta_2]\to(0,\zeta_2]$,  then each semi-wavefront is in fact a wavefront.
\end{pro}

In the particular case when $g$ is monotone, Proposition \ref{pro} says that semi-wavefronts for non-local equation (\ref{nle0}) are wavefronts, indeed these are monotone wavefronts (see Remark \ref{Mono} ). The problem in determining the condition for which $\kappa$ is a global attractor for  $g:(0,\zeta_2]\to(0,\zeta_2]$  was dealt in \cite{YZ2} where the following condition characterizes this globalness property.   

 \vspace{3mm}
 
{\rm \bf (G)} The application $g^2$ has a unique fix point $\kappa$ on $(0,\zeta_2]$.  
  
  \vspace{3mm}

\noindent In this sense, under condition {\bf (G)} and an additional hypothesis on $K$ (which can be dropped by Proposition \ref{pro}) the authors in \cite{YZ} have stated the existence of minimal speed for the existence of wavefronts.

 In another cases it is also possible to determinate whether a semi-wavefronts is actually a wavefront. For instance, since by  \cite[Remark 12]{GPT} we have $[m, M]\subset g([m, M])$  where $m=\liminf_{s\to+\infty}\phi_c(s)$ and $M=\limsup_{s\to+\infty}\phi_c(s)$ for each semi-wavefront $\phi_c$ such that $c\geq c_*^+$ (a similar conclusion is obtained for $c\leq c_*^-$) it is easy to check that if $|g_{|_{[m, M]}}|_{Lip}<1$ and $\kappa\in[m, M]$ then $\phi_c$ is a wavefront under the condition {\bf (M)}. This situation occurs in our two  following stability results. 
 

 Finally, note that the case $c_*^+c_*^-\geq 0$ is possible. For instance, by taking $K(s)=e^{-(s+\rho)^2}/\sqrt{4\pi}$, with $h=2$, $g'(0)=2$ and $\rho=5$, the authors in \cite[page 16]{GPT} show that  $c_*^-=2.7$ and $c_*^+=0,7...$. Thus, in this case the equation (\ref{nle0}) has {\it stationary semi-wavefronts} (for $c=0$) and {\it backwards traveling fronts}  (for $c\in(0, c_*^-)$).

\subsection{Local stability of d-dimensional planar waves} 
\vspace{3mm}

Following notation of Subsection 2.1 we denote $E_c(\lambda)=q_{\lambda}+p_{\lambda}$. Also, for some $\lambda\in\R$ we define the bounded weight function  $$\eta_{\lambda}(z):=\min\{1, e^{\lambda z}\},$$ 
  
 and define the space $$\mathcal{B}_{\lambda}^{r,p}:=\{u:\R^{d}\to\R, \ \hbox{such that}\  |u|_{W_{\lambda}^{r,p}}:=||\eta_{\lambda}u||_{W^{r,p}(\R^d)}<\infty \}
$$

  
\begin{thm}\label{aes1} 
  Suppose ${\bf (L)}$ and $\rho_{\epsilon}:=|g_{|_{[\kappa-\epsilon,\kappa+\epsilon]}}|_{Lip}<1$ for some $\epsilon>0$. Moreover, suppose that for some $c\in\R$, $\nu\in\mathbb{S}^1$ and $\phi_c(\nu\cdot z)$ solution of (\ref{este}) there exist $z_{\epsilon}\in\R$ such that 
  \begin{eqnarray*}
  \phi_c(\nu\cdot z)\in[\kappa-\epsilon/2,\kappa+\epsilon/2]\quad  \hbox{for all} \ z\in\{y\in\R^d: \nu\cdot y\geq z_{\epsilon}\}.
  \end{eqnarray*}
Then,  if  for some $\lambda\in\R^d$ such that $K\in L^1_{\lambda}$,  the non-negative initial datum $u_0\in C^{0,\alpha}_h$ satisfies
  \begin{eqnarray}\label{inequ21}
|u_0-\phi_c|_{\mathcal{B}^1_{h,\lambda}},\ |u_0-\phi_c|_{L_{h}^{\infty}}\leq \epsilon C_{\epsilon}, \ 
\end{eqnarray}   
   for certain  $C_{\epsilon}\in(0,1/2]$, then the following assertions are true

\begin{itemize}
\item[(i)]  If $E_c(\lambda)<0$ then for each $0<\gamma_*\leq \gamma_{\lambda}$ satisfying $\rho_{\epsilon}e^{\gamma_* h}<1-\gamma_*$ we have

\begin{eqnarray}\label{Inqu21}
|u(t,z)-\phi_c(\nu\cdot z)|\leq \frac{\epsilon}{2}e^{-\gamma_{*} t}\quad \forall (t,z)\in[-h,\infty)\times\R^d.
\end{eqnarray}

\item[(ii)]  If $E_c(\lambda)=0$  then there exists $\delta^*=\delta^*(\rho_{\epsilon})>1+h$ such that
\begin{eqnarray}\label{ls}
|u(t,z)-\phi_c(\nu\cdot z)|\leq \frac{\epsilon}{2(t+\delta^*)^{d/2}}  \quad \forall (t,z)\in[-h,\infty)\times\R^d.
\end{eqnarray}
\end{itemize}
  
\end{thm}


It is instructive to compare Theorem \ref{aes1} with a work of Gallay \cite{G} about the local stability of critical wavefronts to a local equation with $h=0$. Note that in \cite{G} the perturbation is additionally weighted with quadratic function in the trivial equilibrium and the exponential convergence to the positive equilibrium is assumed. Although, in this subspace considered by Gallay the rate the convergence is as $O(t^{-3/2}).$ Otherwise, note that for non-critical semi-wavefronts the convergence rate depends on the weighted space where the perturbation is taken attaining an algebraic convergence rate when the perturbation is in $C([-h,0],L^1(\R))$ with weight $\eta_{\lambda_j(c)}$, $j=1,2$, and an exponential convergence rate if $\lambda\in(\lambda_1(c),\lambda_2(c))$.     


\bigskip

\subsection{Global stability of wavefronts on the line}
 \vspace{3mm}
 
 In this section we take $d=1$ and give a global result in the sense that the wavefronts are attractors for the following class of initial data
 \vspace{3mm}
 
{\rm {\bf (IC)} The continuous initial datum $u_0:[-h,0]\times\R\to\R_{\geq 0}$ to (\ref{nle}) is a bounded function
 \indent and there are $\sigma>0$ and $z_0\in\R$ such that : 
$$
 \pm c\geq\pm c_*^{\pm}\ \hbox{implies} \ u_0(s,\pm z)\geq\sigma\ \hbox{for all} \ s \in[-h,0]\ \hbox{and} \  z\geq  z_0.
$$}
\bigskip

\noindent We note that in Theorem \ref{aes1} it is necessary $\epsilon<\kappa$ and $\phi_c(z)\geq \kappa-\epsilon/2$ for  $z\geq z_{\epsilon}$, therefore an initial datum satisfying the condition (\ref{inequ21}) meets the condition ({\bf IC}) with $\sigma=\kappa-\epsilon$ and $z_0=z_{\epsilon}.$

\bigskip

\noindent Denote $M_g:=\max_{u\in[0,\kappa]}g(u)$, $m_g:=\min_{u\in[\kappa,M_g]}g(u)$ and $I_g:=[m_g,M_g]$. Also we define the following weighted Sobolev space  
\small $$
 W_{h,\lambda}^{r,p}:=C([-h,0], W^{r,p}_{\lambda})=\{u:[-h,0]\to W^{r,p}_{\lambda} \ \hbox{continuous with norm}\ |u|_{W^{r,p}_{h,\lambda}}=\max_{s\in[-h,0]}|u(s)|_{\mathcal{B}_{\lambda}^{r,p}} \}  
  $$\normalsize 
  
\vspace{3mm}
\begin{thm}\label{aes}
Suppose  ${\bf (M)}$,${\bf (K)}$ and $\rho:=|g_{|_{I_g}}|_{Lip}<1$. If $c\in\mathcal{C}$ then each semi-wavefront $\phi_c$ is actually a wavefront. Moreover, for each $\pm\lambda_c\geq\pm\lambda_{1}(c)$,  
 and $u_0\in C^{0, \alpha}_h$ satisfying the condition {\bf (IC)} and $u_0-\phi_c\in L^{1}_{h,\lambda}\cap \mathcal{B}^{\infty}_{h,\lambda}$  the following assertions are true

\begin{itemize}
\item[(i)] If $E_c(\lambda_c)<0$ then for any  $0< \gamma_{*}\leq\gamma_{\lambda}$ satisfying $\rho e^{\gamma_*h}< 1-\gamma_*$ there exists  $C=C(g,c,u_0)>0$ such that

\begin{eqnarray}\label{Inqu2}
|u(t,z)-\phi_c(z)|\leq Ce^{-\gamma_{*} t}\quad \forall (t,z)\in[-h,\infty)\times\R.
\end{eqnarray}

\item[(ii)]  If  $E_c(\lambda_c)=0$  then there exists $C>0$ such that
\begin{eqnarray}\label{ls}
|u(t,z)-\phi_c(z)|\leq \frac{C}{\sqrt{t}}  \quad \forall (t,z)\in(0,\infty)\times\R.
\end{eqnarray}
\end{itemize}
\end{thm}
 
 \begin{rem}[Speeds selection problem]\label{sspeed}
 Consider $\pm c\geq \pm c_{*}^{\pm}$ and  $v_0\in C^{0,\alpha}_h$ an initial datum to (\ref{nle0}) satisfying {\bf (IC)} and in the form $v_0(s,x)=-Az^{j_c}e^{\lambda_{1}(c) x}+O(e^{(\lambda_1(c)\pm\epsilon)x})$ for some $A,\epsilon\in\R_+$ and where $j_c=0,1$ and $j_c=1$ if and only if $c=c_*^+$ or $c=c_*^-$ then for each $\beta\in(0,\kappa)$ the associated level set for $v(t,\cdot)$ is asymptotically propagated with speed $c$.  More precisely, by (\ref{rc1}) and (\ref{Inqu2})-(\ref{ls}) for large $t$ the set $\{x\in\R: v(t,x)= \beta\}$ is not empty and, for instance, if $c\geq c_*^+$ is lower bounded therefore it has a infimum $m(t)$, then by evaluating in  (\ref{Inqu2})-(\ref{ls}) at  $z=m(t)+ct$ we necessarily have $m(t)+ct$ is asymptotically bounded in the variable $t$ so that $|c+m(t)/t|=O(1/t)$, i.e., $m(t)$ is propagated with speed $-c$. Also,  note that if  $c_*^+<0$ and $c\in(c_*^+,0)$ the level set will move to $+\infty$ contrary to symmetric case. A similar situation occurs when $c \leq c_*^-$. 
 \end{rem}
 \begin{coro}[Global stability en Sobolev spaces ] \label{sobolev}
 Assume that $u_0-\phi_c\in W^{1,p}_{h,\lambda}$ for $1\leq p<\infty$. Then, $E_c(\lambda_c)<0$ implies (\ref{Inqu2}) and $E_c(\lambda_c)=0$ implies (\ref{ls}).
  \end{coro}

This result includes the classic Fisher-KPP model when $h=0$ and $K$ is the Dirac function. We also have the following result for non-local Nicholson's model
   \begin{coro}[Nicholson's model]\label{cnich}
 Suppose that  $p/\delta\in[1,e^2)$ in (\ref{nich}). If $\phi_c$ is a wavefront with speed $c\in\mathcal{C}$ to (\ref{nich}) then $\phi_c$ is either globally algebraically stable in $\mathcal{B}^1_{h,\lambda_1}$ or globally exponentially stable in $\mathcal{B}^1_{h,\lambda}$  whenever $E_c(\lambda)<0$,  in the sense of Theorem \ref{aes}. 
  \end{coro}
  
 
 For a local version of (\ref{nich})  in \cite[Theorem 2.3]{GT} it was demonstrated that  for $p/\delta\in[e,e^2)$ this equation has non-monotone wavefronts with speed arbitrarily large. Under this restriction on the parameters $p$ and $\delta$, Solar and Trofimchuk have demonstrated the global stability of non-critical wavefronts for the local Nicholson equation \cite[Corollary 3]{STR1}. Thus, the global stability of critical wavefronts for local equations  in Corollary \ref{cnich} is a complement to the result obtained in  \cite{STR1}.
  


 
 
\section{A Regularity Result}
 We start  giving a result on the persistence of disturbances in the underlying space for the following equation  

\small\begin{eqnarray}\label{ale}
\dot{u}(t,z)=\Delta u(t,z)+d_1\cdot\nabla u(t,z)+d_2u(t,z)+\int_{\R^d}K(z-y)d_3(t,y)u(t-h,y)dy \quad z\in\R^d.
\end{eqnarray}\normalsize

\begin{pro}\label{prop}
Suppose  $d_1\in\R^d$, $d_2\in\R$, $d_3\in L^{\infty}(\R_+\times\R^d)$ and the function $u:[-h,+\infty)\times\R^d\to\R$ is solution of (\ref{ale}) on $\R_+\times\R^d$.
 If   for some $\lambda'$ such that $K\in L^1_{\lambda'}$ the initial datum $u_0$ holds $u_0\in L_{h,\lambda'}^p$ some $1\leq p\leq \infty$ 
 then we have the following estimate for the associate solution $u(t,z)$ to (\ref{ale})
 \begin{eqnarray}\label{estf}
 |u_{kh}(\cdot,\cdot)|_{L^p_{h,\lambda'}}\leq \theta^{k+1}  |u_0|_{L_{h,\lambda'}^p}\quad\hbox{for} \quad k=1,2,...
  \end{eqnarray}

 for some $\theta=\theta(\lambda')>1$.
  \end{pro} 
 \begin{proof}
 By making the change of variables $\bar{u}(t,z):=u(t,z)e^{-\lambda'\cdot z}$ the equation (\ref{ale}) is transformed to
 \small\begin{eqnarray}\label{ale1}
\dot{\bar{u}}(t,z)=\Delta\bar{u}(t,z)+d_1'\cdot\nabla\bar{u}(t,z)+d_2'\bar{u}(t,z)+\int_{\R}K'(z-y)d_3(t,y)\bar{u}(t-h,y)dy, 
\end{eqnarray}\normalsize
where $d_1'=2\lambda'+d_1$, $d_2'=|\lambda'|^2+d_1\cdot\lambda'+d_2$ and  $K'(y)=K(y)e^{-\lambda'\cdot y}$. 
Next,  by the change of variable $\bar{\bar{u}}(t,z):=\bar{u}(t,z-d_1't)e^{d_2' t}$ the equation (\ref{ale1}) is reduced
to inhomogeneous heat equation, 
\begin{eqnarray}\label{He}
\dot{\bar{\bar{u}}}(t,z)=\Delta\bar{\bar{u}}(t,z)+f(t,z) \quad \hbox{for all} \quad (t,z)\in\R_+\times\R^d
\end{eqnarray}
where $$f(t,z)=e^{-d'_2h}\int_{\R^d}K'(y)d_3(t,z-d'_1t-y)\bar{\bar{u}}(t-h,z-y-d'_1h)dy.$$
Since $f(t,\cdot)\in {\it L}^{1}$ for $t\in(0,h]$, by denoting $\Gamma_t$ the $d$-dimensional heat kernel  we have
\begin{eqnarray}\label{heat}
\bar{\bar{u}}(t)= \Gamma_t\ast \bar{\bar{u}}(0)+\int_0^t\Gamma_{t-s}\ast f(s,\cdot)ds
\end{eqnarray}
So that, for $t\in(0,h]$
\begin{eqnarray*}
||\bar{\bar{u}}(t)||_{L^p}&\leq& ||u(0)||_{L^p}+h\sup_{s\in[-h,0]}||f(s,\cdot)||_{L^p}\\
&\leq& (1+he^{-d'_2h}||d_3||_{L^{\infty}}||K'||_{L^1})||\bar{\bar{u}}_0||_{L^p_{h}},
\end{eqnarray*}

\noindent therefore if we multiply by $e^{-d'_2t}$ the last inequality then (\ref{estf}) follows for $k=1$ by taking 
 $$\theta:=e^{2h|d_2'|} [1+he^{-d'_2h}||d_3||_{L^{\infty}}||K'||_{L^1}].$$
Analogously, by using $u(t+h,\cdot), u(t+2h,\cdot)...$, with $t\in(0,h]$,  for the intervals $[h,2h],[2h,3h]...$ we obtain  (\ref{estf}) for $k=2,3...$

 \end{proof}
\begin{pro}\label{corop}
Suppose that $g$ is globally Lipschitz continuous  and $K\in L^{1}_{\lambda'}\cap L^{1}_{\lambda}$ for some $\lambda',\lambda\in\R$. If  the initial datum $u_0\in L^{\infty}_{h,\lambda'}\cap C^{0, \alpha}_{h,\lambda}$  then there exist a unique solution $u(t,z)$ to  the nonlinear equation (\ref{nle}) and the solution $u(t,z)$ satisfies the estimation (\ref{estf}) and $u(\cdot+kh, \cdot)\in L^{\infty}_{h,\lambda'}\cap C^{0,\alpha}_{h,\lambda}$ for all $k=0,1,2,...$
\end{pro}
\begin{proof}
We consider  the  Cauchy problem associated to (\ref{He})  for $(t, z)\in[0,h]\times\R^d$ with $d_3(t,z)=g(u(t-h, z-\nu ch))/u(t-h,z-\nu ch)$. Since $f(\cdot-h,\cdot)\in  C^{0, \alpha}_{h}$ and $\bar{\bar{u}}_0\in L^{\infty}_h$ we conclude there exists a unique solution $u(t,z)$ satisfying the Cauchy problem associated (\ref{He}) on $(t, z)\in[0,h]\times\R^d$ (see, e.g., \cite[Chapter 1, Theorem 12 and Theorem16]{AF}) and by (\ref{heat}) we also have $\bar{\bar{u}}(\cdot+h,\cdot)\in C^{0,\alpha}_h$, so that $u(\cdot+h, \cdot)\in C^{0,\alpha}_{h,\lambda}$ and by Proposition \ref{prop} we also have $u(\cdot+h, \cdot)\in L^{\infty}_{h,\lambda'}$. Thus, by repeating this process on the intervals $[2h, 3h], [3h,4h]...$ we obtain a solution $u(t, z)$ for the Cauchy problem of (\ref{He}) on $\R_+\times\R^d$. 

 Finally, by applying  Proposition \ref{prop}  we obtain (\ref{estf}) for $u(t,z)$.
\end{proof}
\begin{rem}\label{classic}
Note that the same procedure used in Proof of Proposition \ref{corop} can be applied to equation heat only requiring that $u_0\in L^{\infty}_{h,\lambda'}\cap C^{0,0}_{h,\lambda}$ which gives a mild solution $u(t,z)$ of (\ref{nle}) for all $(t,z)\in\R_+\times\R^d$. 
\end{rem}
\begin{lem}\label{lemre}
Assume $d_1,d_2$ and $d_3$ satisfy the hypothesis of Proposition \ref{prop}. Consider $u:[-h,h]\times\R\to\R$  satisfying  equation (\ref{ale}) on $(t,z)\in[0,h]\times\R$. If $u_0\in L^{p}_{h,\lambda'}$  for some $\lambda'\in(a,b)$ then
 \begin{eqnarray}\label{lemest} 
|u_{z_i}(t,\cdot))|_{L^p_{\lambda'}} &\leq &  (\frac{\theta_1}{\sqrt{ t}} + \sqrt{t} \theta_2)|u_0|_{L^p_{h,\lambda'}}\quad t\in (0,h] \ \hbox{and} \  i=1,...,d
\end{eqnarray}
for some positive numbers $\theta_1=\theta_1(\lambda')$ and $\theta_2=\theta_2(\lambda')$.
\end{lem}
\begin{rem}\label{Deri}
By using Proposition \ref{prop}, $u_0\in L_{h,\lambda'}^p$ implies $u(\cdot+h,\cdot)$, $u(\cdot+2h,\cdot)...\in L^{p}_{h,\lambda'}$,  therefore Lemma \ref{lemre} implies that for $k=1,2,3,...$ and $t\in[0,h]$ we have $u_{z_i}(t+kh,\cdot)\in L^{p}_{\lambda'}$ with  $i=1,...,d$. Thus for $k=1,2,3...$ Proposition \ref{prop} implies
 \begin{eqnarray}\label{Lemest} 
|u_{z_i}(t+kh,\cdot))|_{L^p_{\lambda'}} \leq  (\frac{\theta_1}{\sqrt{ t}} + \sqrt{t} \theta_2)\theta^{k+1} \ |u_0|_{L^p_{h,\lambda'}}\quad t\in (0,h] \ \hbox{and} \  i=1,...,d.
\end{eqnarray}
\end{rem}
\begin{proof}
If $t>0$ from (\ref{heat}) it follows
 
 \small\begin{eqnarray}\label{deri}
\bar{\bar{u}}_{z_i}(t,z)=\int_{\R^d}\frac{(z_i-y_i)e^{-(z-y)^2/4t}}{2^{d+1}t(\pi t)^{d/2}}\bar{\bar{u}}(0,y)dy+ \int_{0}^t\int_{\R^d}\frac{(z_i-y_i)e^{-(z-y)^2/4(t-s)}}{2^{d+1}(t-s)[\pi (t-s)]^{d/2}} f(s,y)dyds
 \end{eqnarray}\normalsize
therefore for each $t\in (0,h]$ we get
\begin{eqnarray*} 
|\bar{\bar{u}}_{z_i}(t,z)|_{L^p} &\leq & \frac{|u(0)|_{L^p}}{\sqrt{\pi t}} \int_{\R^d}|y_i|e^{-|y|^2}dy+2\sqrt{\frac{t}{\pi}}\int_{\R^d}|y_i|e^{-|y|^2}dy\  |f(\cdot,\cdot)|_{L^p_h} \\
&\leq&  [\frac{|u(0)|_{L^p}}{\sqrt{\pi t}} +2 \sqrt{\frac{t}{\pi}}|d_3|_{L^{\infty}}|K'|_{L^1}e^{|d'_2|h} \bar{\bar{u}}(\cdot,\cdot)|_{L^p_h}]\int_{\R^d}|y_i|e^{-|y|^2}dy\\
&\leq&      [\frac{1}{\sqrt{ t}} +2 \sqrt{t}|d_3|_{L^{\infty}}|K'|_{L^1}e^{2|d'_2|h}]\frac{|u_0|_{L^p_h}}{\sqrt{\pi}},
\end{eqnarray*} 
which implies (\ref{lemest}) by taking $\theta_1=e^{|d'_2|h}/\sqrt{\pi}$ and $\theta_2= 2|d_3|_{L^{\infty}}|K'|_{L^1} \ e^{2|d'_2|h}/\sqrt{\pi}$.

 \end{proof}
 
\begin{pro}[$L^p$-Regularity] \label{reg}
Suppose $u$ satisfies the hypothesis of Proposition \ref{prop} with $\lambda'=0$. If $d_3(t,\cdot)\in L^{\infty}(\R_+;W^{k,\infty}(\R))$ for some $k\in\Z_+$ then
 \begin{eqnarray}\label{regu}
 D^k u(t,\cdot)\in L^{p}(\R) \quad \hbox{for all} \quad t\in((k-1)h,+\infty),
 \end{eqnarray}
 uniformly, in norm, on compact sets of $((k-1)h,+\infty)$.
\end{pro}
\begin{proof}
If $k=1$ by Lemma \ref{lemre} we have $u_{z_i}(t,z)\in L^{p}$, for each $t\in(0,h]$ uniformly (in norm) on compacts. Moreover, by Remark \ref{Deri}, $u_{z_i}(t,\cdot)\in L^{p}(\R)$, for each $t\in(0,+\infty)$, uniformly (in norm) on compacts. In particular, if  $\ T>h$ then $u_{z_i}(t+T,\cdot)\in L_{h}^p$,  therefore we analogously conclude $u_{z_jz_i}(t, \cdot)\in L^{p}$,  for each $t\in(h,+\infty)$, uniformly (in norm) on compacts. The same argument is applied for $k=3,4....$ in order to obtain (\ref{regu}). 
\end{proof}

\section{Proof of Theorem \ref{st} and Theorem \ref{aes1}}

 \noindent We denote the Fourier transform of $u:\R^d\to\R$ by 
 $$
\hat{u}(z)= \frac{1}{(2\pi)^d}\int_{\R^d}e^{-{\bf i}z\cdot y}u(y)dy
$$ 
 Next, we define the function $l:\R^d\to\R_{\geq 0}$ by the equation 
\begin{eqnarray}\label{lamb}
l_{\lambda}(\zeta)=-|\zeta|^2 +p_{\lambda}+L_ge^{-\lambda\cdot\nu ch}|\widehat{\xi_{\lambda} K}(\zeta) |e^{-hl_{\lambda}(\zeta)},
\end{eqnarray}
 
Now, we will estimate the function $l_{\lambda}(\zeta)$. For $\epsilon_h=1/[1+hq_{\lambda}e^{h\gamma_{\lambda}}]>0$ we define the function 
$$\alpha_{h}(\zeta):=-\frac{1}{h}\log(1+h\epsilon_h |\zeta|^2),$$

and we denote by $\hat{q}_{\lambda}(\zeta):= L_ge^{-\lambda\cdot\nu ch}|\widehat{\xi_{\lambda} K}(\zeta) |.$
\begin{lem}\label{loge}
The function $l_{\lambda}$ meets the following inequalities 
\begin{eqnarray}\label{gauss}
-\epsilon_h|\zeta|^2-\gamma_{\lambda}\leq l_{\lambda}(\zeta)\leq \alpha_{h}(\zeta)-\gamma_{\lambda} \quad \hbox{for all}\ \zeta\in\R.
\end{eqnarray}
\end{lem}

\begin{rem}
In the local case (when $\hat{k}$ is formally a constant q) we have $e^{l(\zeta)}\sim -q/\zeta^{2}$ (see \cite[Lemma 13]{S}) but in the non local case, because of Riemann-Lebesgue Lemma, the estimations for $l(\cdot)$ can be improved.
\end{rem}
\begin{proof}
Let us denote $\beta(\zeta)=l_{\lambda}(\zeta)-\alpha_h(\zeta)+\gamma_{\lambda}$. Then $\beta(\zeta)$ satisfies the following equation 
$$
 \beta(\zeta)=-|\zeta|^{2}+\frac{1}{h}\log(1+h\epsilon_h|\zeta|^2)+\gamma_{\lambda}+p_{\lambda}+\hat{q}_{\lambda}(\zeta)e^{h\gamma_{\lambda}}(1+h\epsilon_h|\zeta|^2)e^{-h\beta(\zeta)}.
$$
 From Lemma \cite[Lemma 12]{S}  we have that $\beta(\zeta)\leq 0$ if and only if:
 \begin{eqnarray}\label{desl}
|\zeta|^{2}-\frac{1}{h}\log(1+h\epsilon_h|\zeta|^2)-\gamma_{\lambda}-p_{\lambda}\geq \hat{q}_{\lambda}(\zeta)e^{h\gamma_{\lambda}}(1+h\epsilon_h|\zeta|^2).
\end{eqnarray}
Now, by using $\log(1+x)\leq x$, for all $x\geq 0$, then in order to obtain (\ref{desl}) it is enough to have
\begin{eqnarray*}
|\zeta|^2-\epsilon_{h}|\zeta|^2-\gamma_{\lambda}-p_{\lambda}\geq \hat{q}_{\lambda}(0)e^{h\gamma_{\lambda}}(1+h\epsilon_h|\zeta|^2)\quad\hbox{for all}\quad\zeta\in\R\\
\iff (1-\epsilon_{h}-q_{\lambda}h\epsilon_h e^{h\gamma_{\lambda}})|\zeta|^2-\gamma_{\lambda}-p_{\lambda}- q_{\lambda}e^{h\gamma_{\lambda}}= 0\quad\hbox{for all}\quad\zeta\in\R\\
\end{eqnarray*} 
This proves (\ref{gauss}).  

\end{proof}

\underline{Proof of Theorem \ref{st}} 
 
 Note that  by Proposition \ref{corop},  $u(t,\cdot)$ and $\psi(t,\cdot)$ exist uniquely in $L^{\infty}_{h,\lambda}$. 
 
  Then, by making the following change of variable $\tilde{u}(t,z)=u(t,z)e^{-\lambda\cdot z}$ we have
\small\begin{eqnarray*}
\dot{\tilde{u}}(t,z)=\Delta\tilde{u}(t,z)+(2\lambda-c\cdot \nu)\nabla\tilde{u}(t,z)+p_{\lambda}\tilde{u}(t,z)+e^{-\lambda \cdot z}[K\ast g(e^{\lambda\cdot(\cdot-ch\nu)}\tilde{u}(t-h,\cdot-ch\nu))](z).
\end{eqnarray*}\normalsize


Now, if  we denote the linear operator $$\mathcal{L}_0\delta(t,z):=\Delta\delta(t,z)+(2\lambda-c\nu)\cdot\nabla\delta(t,z)+p_{\lambda}\delta(t,z)-\dot{\delta}(t,z),$$

and  $\delta_{\pm}(t,z):=\pm[\tilde{u}(t,z)-\tilde{\psi}(t,z)]-r(t,z)$  then for $(t,z)\in[0,h]\times\R^d$ we have

\small\begin{eqnarray*}
 (\mathcal{L}_0)\delta_{\pm}(t,z)=\pm e^{-\lambda z}K\ast [g(e^{\lambda(\cdot-ch)}\tilde{\psi}(t-h,\cdot-ch))- g(e^{\lambda(\cdot-ch)}\tilde{u}(t-h,\cdot-ch))](z)-\mathcal{L}_0r(t,z)\\
 \geq-L_ge^{-\lambda ch}|K\xi_{\lambda}\ast [\tilde{\psi}(t-h,\cdot-ch)-\tilde{u}(t-h,\cdot-ch)](z)|-\mathcal{L}_0r(t,z).
\end{eqnarray*}\normalsize
Because of $|\tilde{\psi}(t-h,z)-\tilde{u}(t-h,z)|\leq r(t-h,z)$ for $(t,z)\in[0,h]\times\R$ then 
\begin{eqnarray}
\geq -L_ge^{-\lambda ch}(K\xi_{\lambda}\ast r)(t-h,z-ch)-\mathcal{L}_0r(t,z).
\end{eqnarray}

\noindent Now, as $w(t,z)=u(t,z)-\psi(t,z)$ satisfies 
 $$  
 \dot{w}(t,z)=\Delta w(t,z)-c\nu\cdot\nabla w(t,z)-w(t,z)+\int_{\R^d}K(z-ch-y)d_3(t,y)w(t-h,y)dy,  
 $$  
where $d_3(t,y)=[g(u(t-h,y))-g(\psi(t-h,y))]/[u(t-h,y)-\psi(t-h,y)]$, by  Proposition \ref{prop}, with $\lambda'=\lambda$, and
  Phragm\`en-Lindel\"of principle \cite[Chapter 3, Theorem 10]{PW} we obtain
$$
\pm[ \tilde{u}(t,z)-\tilde{\psi}(t,z)]\leq r(t,z)\quad \quad \hbox{in}\ [0,h]\times\R.
$$ 
By repeating the same process for the intervals $[h,2h],[2h,3h]...$ we conclude

\begin{eqnarray}\label{L1}
\pm[ \tilde{u}(t,z)-\tilde{\psi}(t,x)]\leq r(t,z)\quad \quad \hbox{in}\ [-h,+\infty)\times\R.
\end{eqnarray}
Now, we globally estimate the function $r$. Next, by Proposition \ref{reg} we have  $r,r_{z_i},r_{z_iz_i}\in L^1(\R)$ for all $t>h$. Then, by applying  Fourier's transform to (\ref{lde}) we obtain
\small\begin{eqnarray*}
\hat{r}_t(t,z)=(-|z|^2+i (2\lambda-c\nu)\cdot z+p_{\lambda})\hat{r}(t,z)+L_ge^{-\lambda\cdot\nu ch}\widehat{\xi_{\lambda} K}(z)e^{-ichz\cdot\nu}\hat{r}(t-h,z),
\end{eqnarray*}\normalsize
 for all $(t,z)\in(2h,+\infty)\times\R$. So, due to \cite[Lemma 11]{S}, by using $l_{\lambda}+\gamma_{\lambda}\leq 0$, we get
\begin{eqnarray}\label{ee} 
e^{\gamma_{\lambda}t}|\hat{r}(t,z)|\leq e^{(l_{\lambda}(z)+\gamma_{\lambda})t}\quad \hbox{for all} \quad (t,z)\in(2h,+\infty)\times\R,
\end{eqnarray}
 and by Lemma \ref{loge} we have
\begin{eqnarray}\label{log1}
l_{\lambda}(z)+\gamma_{\lambda}\leq -\frac{1}{h}\log(1+h\epsilon_h|z|^2)\quad\forall z\in\R^d.
\end{eqnarray} 
Finally, due to (\ref{ee}) and (\ref{log1}) we obtain that $\hat{r}(t,\cdot)\in L^1(\R)$ for $t>h(d+1)/2$ and by using  Fourier's inversion formula  we have (in this computation we replace $|\cdot|_{L^1}$ by $|\cdot|$ )
\begin{eqnarray*}
|r(t,z)|&\leq& \frac{1}{(2\pi)^d}\int_{\R^d} |\hat{r}(t,y)|dy\leq \frac{|r_0|}{(2\pi)^d}\int_{\R^d} e^{l_{\lambda}(y)t}dy\\
&\leq& \frac{|r_0|}{(2\pi)^d}e^{-\gamma_{\lambda} t}\int_{\R^d} \frac{dy}{(1+h\epsilon_h|y|^2)^{\frac{t}{h}}}=\frac{|r_0|e^{-\gamma_{\lambda} t}}{(2\pi)^d}\int_{0}^{+\infty}\int_{\partial B(0,r)}[1+h\epsilon_h r^2]^{-\frac{t}{h}}dSdr\\
&=& \frac{|r_0|e^{-\gamma_{\lambda} t}}{2^{d-1}\pi^{d/2}\Gamma(d/2)}\int_{0}^{+\infty}\frac{r^{d-1}dr}{[1+h\epsilon_h r^2]^{t/h}}=\frac{|r_0|e^{-\gamma_{\lambda} t}}{2^{d-1}\Gamma(d/2)(\pi\epsilon_h t)^{d/2}}\int_{0}^{+\infty}\frac{r^{d-1}dr}{[1+\frac{r^2}{t/h}]^{t/h}}\\
&\leq& \frac{|r_0|e^{-\gamma_{\lambda} t}}{2^{d-1}\Gamma(d/2)(\pi\epsilon_h t)^{d/2}}\int_{0}^{+\infty}r^{d-1}e^{-r^2}dr=  \frac{|r_0|e^{-\gamma_{\lambda} t}}{2^{d-1}\Gamma(d/2)(\pi\epsilon_h t)^{d/2}}\ \frac{1}{2}\Gamma(d/2)\\
&=&  \frac{|r_0|e^{-\gamma_{\lambda} t}}{2^{d}(\pi\epsilon_h t)^{d/2}}
\end{eqnarray*}
\underline{Proof of Theorem \ref{aes1}}

\begin{itemize}

  \item[(i)] 
  
  

 Note that by (\ref{rc1}) we get
\begin{eqnarray}\label{A11}
 e^{-\lambda z}|u(t,z)-\phi_c(z)|\leq \frac{|r_0|_{L^1_{h,\lambda}}}{A_{\lambda}t^{d/2}} e^{-\gamma_{\lambda} t}  \quad \ \forall \ t>2h, z\in\R
 \end{eqnarray}

  Now, by Proposition \ref{prop} we can take $r_0$, which we will fix below, with $|u_0-\phi_c|_{L^{\infty}_h}<<\epsilon/2$ such that $|u(t,\cdot)-\phi_c|_{L^{\infty}}\leq \epsilon/2$ for all $t\in [0, 3h]$.  Note that the last inequality implies
  \begin{eqnarray}\label{neik}
  u(t,z)\in[\kappa-\epsilon, \kappa+\epsilon]\quad \hbox{for all} \quad (t,z)\in[0,3h]\times\{\nu\cdot z\geq z_{\epsilon}\} 
  \end{eqnarray}

  Then, we consider a function $r:[-h,+\infty)\to\R_+$ given by $r(t):=\frac{\epsilon}{2}e^{-\gamma_*t}$ and  define $\delta_{\pm}(t,z):=\pm[u(t+3h,z)-\phi_c(z)]-r(t)$. So that,  we obtain 
  $$\delta_{\pm}(s,z)\leq 0\ \hbox{for} \ (s,z)\in[-h,0]\times\R^d.$$

    And, if $(t,z)\in[0,h]\times\R^d$  then by (\ref{A11})  and (\ref{neik}) we get  

  \begin{eqnarray}\label{L}
  \mathcal{L}\delta_{\pm}(t,z)=\pm\int_{\R^d}K(z-ch\nu-y)[g(\phi_c(\nu\cdot y))-g(u(t+2h,y))]dy-\mathcal{L}r(t) 
  \end{eqnarray}
  
  \begin{eqnarray*}
&\geq&-[\frac{g'(0)|r_0|_{L^1_{h,\lambda}} e^{-\gamma_{\lambda}(t+2h)}}{A_{\lambda}(t+2h)^{d/2}}\int_{\nu\cdot y\leq z_{\epsilon}}e^{\lambda y}K(z-ch\nu-y)dy\\
&+&\rho_{\epsilon}\int_{\nu\cdot y\geq z_{\epsilon}} K(z-ch\nu-y)r(t-h)dy]-\mathcal{L}r(t)\\
&\geq&-\frac{\epsilon}{2}e^{-\gamma_* t}[\frac{2g'(0)|r_0|_{L^1_{h,\lambda}} e^{-\gamma_{\lambda}2h}}{\epsilon A_{\lambda}(t+2h)^{d/2}}|K|_{L_{\lambda}^1}+\rho_{\epsilon}e^{\gamma_*h}-1+\gamma_*]\end{eqnarray*}
Now, in the last inequality  since $\rho_{\epsilon}e^{\gamma_*h}< 1-\gamma_*$ we can choose $C_{\epsilon}\in(0,1/2]$ small enough such that $|r_0|_{L^1_{h,\lambda}}\leq \epsilon C_{\epsilon} $ implies $\mathcal{L}\delta_{\pm}(t,z)\geq 0$ for all $(t,z)\in[0,h]\times\R^d$,  so that Phragm\`en-Lindel\"of principle implies  $\delta_{\pm}(t,z)\leq 0$ for $(t,z)\in[0,h]\times\R^d.$

Since $|u(t+3h,z)-\phi_c(\nu\cdot z)|\leq r(t)$ for all $(t,z)\in[0,h]\times\R^d$ implies $u(t+3h,z)\in[\kappa-\epsilon/2, \kappa+\epsilon/2]$ for all $(t,z)\in[0,h]\times\{\nu\cdot z\geq z_{\epsilon}\}$ it is possible to repeat the process, by using (\ref{A11}),  for the intervals $[h,2h],[2h,3h]...$ in order to obtain $\delta(t,z)\leq 0$ for all $(t,z)\in[-h,\infty)\times\R^d$.
   
  \item[(ii)] We take  $\delta^*>1+h$ large enough satisfying 
     
  \begin{eqnarray}\label{inq}
 \frac{\rho_{\epsilon}(t+\delta^*)^{d/2}}{(t+\delta^*-h)^{d/2}}+\frac{d}{2(t+\delta^*)}<1\quad t\geq-h.
  \end{eqnarray}    
    
 Now, we consider $r:[-h,+\infty)\to\R_+$ given by $r(t):=\epsilon/2\sqrt{t+\delta^*}$. Next, by Proposition \ref{prop} we can take $r_0$, which we will fix below, such that $|u(t,\cdot)-\phi_c|_{L_{\lambda}^{\infty}}\leq \epsilon/2\sqrt{\delta^*-h}$ for all $t\in[-h,3h]$. So that, if we define   $\delta_{\pm}(t,z):=\pm[u(t+3h,z)-\phi_c(z)]-r(t)$ then we have $$\delta_{\pm}(s,z)\leq 0\quad \hbox{for all} \ (s,z)\in[-h,0]\times\R.$$ And if $(t,z)\in[0,h]\times\R^d$, by  (\ref{A11})  and (\ref{neik}) we get  
 $$
  \mathcal{L}\delta_{\pm}(t,z)=\pm\int_{\R^d}K(z-ch\nu-y)[g(\phi_c(\nu\cdot y))-g(u(t+2h,y))]dy-\mathcal{L}r(t) 
  $$
  $$
\geq-[\frac{L_g|r_0|_{L^1_{h,\lambda_c}}}{A_h\sqrt{t+2h}}\int_{\nu\cdot y\leq z_{\epsilon}}e^{\lambda_c y}K(z-ch\nu-y)dy+\rho_{\epsilon}\int_{\nu\cdot y\geq z_{\epsilon}}K(z-ch\nu-y)r(t-h)dy+\mathcal{L}r(t)]
$$
\begin{eqnarray*}
\geq-\frac{\epsilon}{2\sqrt{t+\delta^*}}[\frac{g'(0)|r_0|_{L^1_{h,\lambda_c}}}{ A_h}\frac{\sqrt{t+\delta^*}}{\sqrt{t+2h}}|K|_{L^1_{\lambda_c}}+\rho_{\epsilon}\frac{\sqrt{t+\delta^*}}{\sqrt{t+\delta^*-h}}-1+\frac{1}{2(t+\delta^*)}]. 
\end{eqnarray*}
 However, by (\ref{inq}) in the last inequality we can choose $|r_0|_{L_{h,\lambda}^1}<<\epsilon/2$ small enough  such that $\mathcal{L}\delta_{\pm}(t,z)\geq 0$ for all $(t,z)\in[0,h]\times\R^d$, so that Phragm\`en-Lindel\"of principle implies  $\delta_{\pm}(t,z)\leq 0$ for $(t,z)\in[0,h]\times\R^d.$ By repeating the process in the intervals $[h,2h],[2h,3h]...$ we obtain $\delta(t,z)\leq 0$ for all $(t,z)\in[-h,\infty)\times\R^d.$ 
 \end{itemize}

\section{Proof of Theorem \ref{aes}}

\subsection{Monotone case}
We begin this section with some results which generalize those founded in \cite{STR1} and \cite{S}. In this section, $g:\R_+\to\R_+$ is a monotone function which is extended linearly and $C^1$ on $(-\infty,0]$.  
\begin{definition}\label{def1}
Continuous function $u_+:  [-h, +\infty) \times \R \to \R$
is called a super-solution for (\ref{nle}), if, for some $z_* \in
\R$, this function  is $C^{1,2}$-smooth in the domains $[-h, +\infty) \times
(-\infty, z_*]$ and $[-h, +\infty) \times [z_*, +\infty)$ and, for every $t >0$,
\begin{equation}  \label{sso} \hspace{-0.5cm} {\mathcal N}u_+(t,z) \geq 0,  \ z \not= z_*,  \ \mbox{while} \
(u_+)_z(t,z_*-)> (u_+)_z(t,z_*+), 
\end{equation}
where the nonlinear operator ${\mathcal N}$ is defined by
\begin{equation} \hspace{-.0cm}
{\mathcal N}w(t,z):= w_{t}(t,z)
-w_{zz}(t,z)+cw_{z}(t,z)+w(t,z)-\int_{\R}K(y)g(w(t-h, z-ch-y))dy. \nonumber
\end{equation}
The definition of  a  sub-solution $u_-$ is similar, with the  inequalities reversed in 
(\ref{sso}). 
\end{definition}
Also, we define the linear operator 
\begin{eqnarray*}
(\mathcal{L}\delta)(t,z):=\delta_{zz}(t,z)- \delta_{t}(t,z)-c\delta_{z}(t,z)-\delta (t,z).
\end{eqnarray*}
\vspace{2mm}

\begin{lem} \label{cl} Suppose that the non-decreasing function $g$ holds {\rm \bf(M)}. Let $u_+, u_-$ be a pair of super- and sub-solutions for equation (\ref{nle}) such that  $|u_\pm(t,z)| \leq Ce^{D|z|}$, $t \geq -h, \ z \in \R$, for some $C, D >0$ as well as
$$
u_-(s,z) \leq u_0(s,z) \leq u_+(s,z), \quad \mbox{for all} \ s \in [-h,0], \ z \in \R.  
$$
Then the solution $w(t,z)$  of  equation (\ref{nle}) with the initial datum $w_0$  satisfies 
$$
u_-(t,z) \leq u(t,z) \leq u_+(t,z) \quad \mbox{for all} \ t \geq -h, \ z \in \R.  
$$
\end{lem}
\begin{proof} In view of the assumed conditions, we have that
$$
\pm(g(u_\pm(t-h,z-ch))- g(u(t-h,z-ch))) \geq 0, \quad t \in [0,h], \ z
\in \R.
$$
Therefore, for all  $(t ,z)\in [0,h]\times\R\setminus\{z_*\}$, the function $\delta_{\pm}(t,z):= \pm(u(t,z)- u_\pm(t,z))$
satisfies  the inequality 
\begin{eqnarray*}
\delta_{\pm}(0,z) \leq 0,\ |\delta_{\pm}(t,z)| \leq 2Ce^{D|z|},
\end{eqnarray*}

\begin{eqnarray}
 \mathcal{L}\delta_{\pm}(t,z) = 
\pm \{{\mathcal N}u_\pm(t,z) - {\mathcal N}u(t,z) +
K\ast [g(u_\pm(t-h, \cdot)) -
 g(u(t-h, \cdot))](z-ch)\} =\nonumber  \\ 
  \pm{\mathcal N}u_\pm(t,z)\pm \int_{\R}K(y)[g(u_\pm(t-h,z-ch-y)) -
 g(u(t-h, z-ch-y))]dy
   \geq 0, \nonumber 
 \end{eqnarray}   
 
 and
 \begin{equation}\label{jp}
\frac{\partial \delta_{\pm}(t,z_*+)}{\partial z}- \frac{\partial \delta_{\pm}(t,z_*-)}{\partial z}=  \pm\left(\frac{\partial u_\pm(t,z_*-)}{\partial z} -\frac{\partial u_\pm(t,z_*+)}{\partial z}\right) > 0. 
\end{equation}
We claim that $\delta_{\pm}(t,z) \leq 0$ for all $t \in [0,h], \ z \in
\R$. Indeed, otherwise there exists $r_0> 0$ such that
$\delta(t,z)$ restricted to any rectangle $\Pi_r= [-r,r]\times
[0,h]$ with $r>r_0$,   reaches its maximal positive value $M_r >0$
at  at some point $(t',z') \in \Pi_r$.

We claim  that $(t',z')$ belongs to the parabolic boundary
$\partial \Pi_r$ of $\Pi_r$. Indeed, suppose on the contrary, that
$\delta(t,z)$ reaches its maximal positive value at some point
$(t',z')$ of $\Pi_r\setminus \partial \Pi_r$. Then clearly $z'
\not=z_*$ because of (\ref{jp}). Suppose, for instance that $z' >
z_*$. Then $\delta(t,z)$ considered on the subrectangle $\Pi=
[z_*,r]\times [0,h]$ reaches its maximal positive value $M_r$ at the
point $(t',z') \in \Pi \setminus \partial \Pi$.  Then the
classical results \cite[Chapter 3, Theorems 5,7]{PW} show that
$\delta_{\pm}(t,z) \equiv M_r >0$ in $\Pi$, a contradiction.

Hence, the usual maximum principle holds for each $\Pi_r, \ r \geq
r_0,$ so that we can appeal to the proof of the
Phragm\`en-Lindel\"of principle from \cite{PW} (see Theorem 10 in
Chapter 3 of this book), in order to conclude that  $\delta_{\pm}(t,z)
\leq 0$ for all  $t \in [0,h], \ z \in \R$.

But then we can again repeat the above argument on the intervals $[h,2h],$ $[2h, 3h], \dots$ establishing that the inequality $u_-(t,z) \leq u(t,z)\leq u_+(t,z),$  $z\in \R,$ holds for all $t \geq -h$.  

\end{proof}

Now, if $g$ meets {\bf(M)} then,  as in \cite[formula (16) and (17)]{STR1}, 
for given $q^* >0,\ q_* \in (0,\kappa)$, there are $ \delta^* < \delta_0$, $\gamma^* >0$  such that
\begin{eqnarray} \label{gg1}
\begin{array}{ll}
g(u)- g(u- qe^{\gamma h}) \leq q(1-2\gamma), \\
(u,q,\gamma)\in \Pi_-= [\kappa-\delta^*,\kappa+\delta^*]\times[0,q_*] \times
 [0,\gamma^*];
\end{array}
\end{eqnarray} 
\begin{eqnarray} \label{gg}
\begin{array}{ll}
g(u)- g(u+ qe^{\gamma h}) \geq - q(1-2\gamma), \\
(u,q,\gamma)\in \Pi_+= [\kappa-\delta^*,\kappa+\delta^*]\times[0,q^*] \times
 [0,\gamma^*].
\end{array}
\end{eqnarray} 

For $c\geq c_*^+$ and a wavefront $\phi_c$ we fix $z^+=z^+(\phi_c)$ such that  $\phi_c(z)\in[\kappa-\delta^*,\kappa+\delta^*]$  for all $z\geq z^+$ and  if  $c\leq c_*^-$ we fix $z^-=z^-(\phi_c)$ such that $\phi_c(z)\in[\kappa-\delta^*,\kappa+\delta^*]$ for all $z\leq z^-$. Also, for $\gamma\in(0,g'(0))$  we define  $b^+_{\gamma}=b^+_{\gamma}(\phi_c)$ and $b^-_{\gamma}=b^-_{\gamma}(\phi_c)$ by
\begin{eqnarray}\label{int}
g'(0)\int_{b^+_{\gamma}-z^+-ch}^{+\infty}K(y)dy=g'(0)\int^{b^-_{\gamma}-z^--ch}_{-\infty}K(y)dy=\gamma e^{-\gamma h}.
\end{eqnarray}


\begin{thm}\label{Sttg}
Suppose that $g$ is non-decreasing function satisfying $(\textbf{M})$, $K$ satisfies (\textbf{K}) and $\gamma\in[0,\gamma^*]$ satisfies (\ref{gg1})-(\ref{gg}). If  for $\pm c>\pm c^{\pm}_*$ and $\pm \lambda_c>\pm\lambda_1(c)$ we have 
\begin{eqnarray}\label{E}
\gamma+p_{\lambda_c}\leq e^{\gamma h}q_{\lambda_c}
\end{eqnarray}
 then 
\begin{eqnarray*}\label{mlem}u_0(s,z)\leq \phi_c(z)+q \eta_{\lambda_c}(z- b), \quad z\in\mathbb{R},\quad s\in [-h,0],
\end{eqnarray*}
with $q \in (0,q^*]$ and $\pm b\geq \pm b^{\pm}_{\gamma}$ implies
 \begin{eqnarray}\label{mlemm}u(t,z)\leq \phi_c(z)+qe^{-\gamma t}
\eta_{\lambda_c}(z- b), \quad z\in\mathbb{R},\quad t\geq -h,
\end{eqnarray}

Similarly, the inequality
\begin{eqnarray}\label{mlemN}\phi_c(z)-q \eta_{\lambda_c}(z- b) \leq u_0(s,z),  \quad z\in\mathbb{R},\quad s\in [-h,0],
\end{eqnarray}
with some $0< q \leq q_*$ and $\pm b\geq \pm b^{\pm}_{\gamma}$  implies
 \begin{eqnarray}\label{mlemmN}
 \phi_c(z)-qe^{-\gamma t} \eta_{\lambda_c}(z- b) \leq u(t,z),
\quad z\in\mathbb{R},\quad t\geq -h,
\end{eqnarray}
Finally, if $K$ has compact support the conclusions above are true for $\pm c\geq c^{\pm}_*$ and $\pm\lambda\geq\lambda_1(c)$ by taking $\gamma=0$ and the weight $\eta_{\lambda_1(c)}(\cdot-b)$ with $\pm b\geq b^{\pm}_0$ for some $b^{\pm}_0\in\R$ .
\end{thm}

\begin{rem}
It is instructive to compare Theorem \ref{Sttg} with \cite{MOZ} for asymptotic stability of non-critical wavefronts. Due to the continuos embedding $H^1_{\eta^2}(\R)\subset C_{\eta}(\R)\cap C^{0,1/2}(\R_+)$ if we take an initial datum $u_0$ like in \cite{MOZ} then $u_0$ is convergent at $+\infty$, so that $u_0(s,+\infty)=\kappa$ uniformly for $s\in[-h,0]$ and therefore for suitable $q\in(0, q_*]$ the initial datum $u_0$ holds (\ref{mlemN}). Also, the weight function $\eta(\cdot-x_0)$ in \cite{MOZ} is defined by $x_0=x_0(\phi_c)$ ($x_0\geq b_{\gamma}^+$ in our case) such that the wavefront $\phi_c$ belongs to a suitable neighborhood of $\kappa$. In our case the number $b_{\gamma}^+$ also depends upon $\gamma>0$ due to the kernel $K$ could have no compact support which constrain us to do the integral small enough in (\ref{int}) (see formula (\ref{I2}) below). So that, when $K$ has compact support, in particular, we get the local stability of the critical wavefronts  which is a generalization of the local case (compare with \cite[Lemma 2]{STR1}). 
\end{rem}

\begin{proof}
Let $c> c^{+}_*$.
Set $u_{\pm}(t,z)=\phi_c(z)\pm qe^{-\gamma
t}\eta_{\lambda_c}(z- b)$. Then, for $t>0$ and $z\in\R\setminus\{b\}$, after a direct calculation we find that 
\begin{eqnarray}\label{subsuper}
\mathcal{N}u_{\pm}(t,z)&=& \pm qe^{-\gamma
t}[-\gamma\eta_{\lambda_c}(z- b)+c\eta'_{\lambda_c}(z- b)-\eta''_{\lambda_c}(z-b)+\eta_{\lambda_c}(z- b)]+\\ \nonumber
&&\int_{\R}K(y)[g(\phi(z-ch-y))-g(u_{\pm}(t-h,z-ch-y))]dy. 
\end{eqnarray}
By (\ref{E}), it is clear that if  $z< b$  it holds that
\begin{eqnarray*}
\pm\mathcal{N}u_{\pm}(t,z)\geq qe^{-\gamma
t}[e^{\lambda_c(z- b)}(-\gamma+c\lambda_c-\lambda_c^{2}+1)-g'(0)e^{\gamma h}\int_{\R}K(y)e^{\lambda_c(z-ch- b-y)}dy]\\
\geq qe^{-\gamma
t+\lambda_c(z- b)}[-\gamma+c\lambda_c-\lambda_c^{2}+1-g'(0)e^{-\lambda_c ch+\gamma h}\int_{\R}K(y)e^{-\lambda_c y}dy]\geq 0.
\end{eqnarray*}
Similarly, if $c<c_-^*$ we have $ \pm\mathcal{N}u_{\pm}(t,z)\geq 0$ for $z>b$ and $t>0$.

Now, for $c> c^+_*$, $z> b$ and $q \in (0, q^*]$, then  
\begin{eqnarray*}
\pm\mathcal{N}u_{\pm}(t,z)&= & qe^{-\gamma
t}[-\gamma+1]\pm[I_1^{\pm}(t,z)+I_2^{\pm}(t,z)], 
\end{eqnarray*}

where 
\begin{eqnarray}\label{I1}
I_1^{\pm}(t,z)= \int_{-\infty}^{-z^+-ch+b}K(y)[g(\phi(z-ch-y))-g(u_{\pm}(t-h,z-ch-y))]dy, 
\end{eqnarray} 

and 
\begin{eqnarray}\label{I2}
I_2^{\pm}(t,z)= \int_{-z^+-ch+b}^{+\infty}K(y)[g(\phi(z-ch-y))-g(u_{\pm}(t-h,z-ch-y))]dy .
\end{eqnarray}

 If  we use formula (\ref{gg}) to estimate $|I_1^{\pm}|$ and (\ref{int}) to estimate $|I_2^{\pm}|$ then for $q>0$ we have $$ 
\mathcal{N}u_+(t,z)\geq qe^{-\gamma t}[1-\gamma-(1-2\gamma)-g'(0)e^{\gamma h}\int_{-z^+-ch+b}^{+\infty}K(y)dy]\geq 0\quad \forall(t,z)\in[-h,\infty)\times[b,+\infty). 
$$
Similarly, 
if  $q \in (0, q_*]$ from (\ref{gg1}) and (\ref{int}) we obtain that
\begin{eqnarray*}
-\mathcal{N}u_{-}(t,z)\geq 0   \quad \forall(t,z)\in[-h,\infty)\times[b,+\infty). 
\end{eqnarray*}
The same arguments are used for $c< c_*^-$ replacing $z^+$ by $z^-$. 
Next, since $$\pm\left(\frac{\partial u_\pm(t,b+)}{\partial z}- \frac{\partial
u_\pm(t,b-)}{\partial z}\right)= -  q\lambda_c e^{-\gamma t} <0,$$ 
we conclude that $u_{\pm}(t,z)$ is a pair of super- and sub-solutions for equation (\ref{nle}). 
So, an application of Lemma \ref{cl}  completes the proof for case $\pm c>\pm c^{\pm}_*$.   

Finally, for the case $c=c_*^{+}$ and $K$ compactly supported we can take $b_0^+$ large enough  in (\ref{I2}) in order to get $I_2^{\pm}=0$ and therefore the proof of (\ref{mlemm}) and (\ref{mlemmN}) with $\gamma=0$ and $\eta_{\lambda_1(c_*^{+})}(\cdot-b), b\geq b_0^+$, is obtained by following the same arguments above. The proof of the case $c=c_*^-$ and $K$ compactly supported is completely analogous.  
\end{proof}
\begin{rem}\label{Mono}[Monotonicity of wavefronts]
Following the abstract setting developed in \cite{LZh},  for $t\geq 0$ we define $Q_t: [0,\kappa]\rightarrow[0,\kappa]$ (this map is well defined since in Lemma \ref{cl} we can take $u_-=0$ and $u_+=\kappa$) as $Q_t(u_0)(x):=u(t,x)$  where $u(t,x)$ is the solution to (\ref{nle}) with initial datum $u_0\in[0,\kappa]$. Next, we note that the hypothesis (K1)-(K5) in \cite{LZh} are trivially satisfied with $\mathcal{K}:=[0,\kappa]\leq L^{\infty}(\R)$ and $O_n$ defined by mean $\varsigma_B$ (see formula (2.1) of \cite[page 861]{LZh}). Also, for $Q=Q_1$ we see that hypothesis (A1), (A2), (A4) and (A6) in \cite{LZh} are trivially satisfied. Then, due to Remark \ref{Deri} (with $\lambda'=0$) for any family of functions $\mathcal{U}$ of $\mathcal{K}=[0,\kappa]\leq L^{\infty}(\R)$, by Arzel\`a-Ascoli Theorem,  we have that $Q_1(\mathcal{U})_I\subset \mathcal{K}$ is relatively compact and therefore (A3) is satisfied. Finally, note that if in (\ref{subsuper}) we take $\phi_c=\kappa$ and $u_-(t,z)=\kappa-qe^{-\gamma t}$, $q\in[0, q_*]$ and $\gamma\in[0, \gamma^*]$ satisfying (\ref{gg1})-(\ref{gg}) we have $\mathcal{N}u_-(t,z)\leq 0$ for all $(t,z)\in[0,+\infty)\times\R$, so that if we take an initial datum $u_0$ such that $u_0>>0$ we can find $q_0\in[0, q_*]$ such that $\kappa-q_0e^{-\gamma^*s}\leq u_0(s,z)$ for all $(s,z)\in[-h,0]\times\R$ and Lemma \ref{cl} implies $\kappa-q_0e^{-\gamma^*t}\leq u(t,z)$ for all $(t,z)\in[0,+\infty)\times\R$ and therefore the condition (A5) in \cite{LZh} is satisfied. Thus, by \cite[Theorem 4.1 and Theorem 4.2]{LZh} all wavefronts of (\ref{nle0}) are monotone if $g$ is monotone. Note that the hypothesis $L_g=g'(0)$ is no mandatory by using these arguments.  
 \end{rem}

\subsection{Attractivity  of an optimal neighborhood of $\kappa$}

\begin{lem}\label{lem}
Let consider  $g_1$ and $g_2$ satisfying {\bf (L)} and  $K_1, \xi_{\lambda}K_2\in L^1(\R)$, some $\lambda\in\R$. Suppose that for some $R\in\R_+\cup\{+\infty\}$: 
\begin{eqnarray}\label{kk}
K_1(y)g_1(u)\leq K_2(y)g_2(u)\ \hbox{for all} \ (y,u)\in \R\times (-\infty,R).
\end{eqnarray} 
 Denote by  $v_1$ and $v_2$ the solutions to (\ref{nle}),  generated by the initial data $v_1^0$ and $v_2^0$ with $Kg = K_1g_1$  and $Kg = K_2g_2$, respectively.
 
  Moreover, if $R<+\infty$ we suppose 
 \begin{eqnarray}\label{v2}
 v_2(t,z)\leq R \ \hbox{for all} \ (t,z)\in[-h,+\infty)\times\R,
 \end{eqnarray}
 while if $R=+\infty$ we suppose
 \begin{eqnarray}
 |v_2^0(s,z)|\leq N e^{\lambda z}\ \hbox{for all}\ (s, z) \in  [-h, 0] \times\R,
  \end{eqnarray}
for some $N>0$.
 
  If  $g_1$  or $g_2$  is a non-decreasing function, then 
  
   \begin{eqnarray}\label{vv}
 0\leq v_1^0(s, z) \leq v_2^0 (s, z)\ \hbox{for all}\ (s, z) \in  [-h, 0] \times\R.
 \end{eqnarray}
 implies  
 \begin{eqnarray*}
 v_1(t, z) \leq v_2(t, z)\ \hbox{for all} \ (t, z) \in \R_+ \times \R.
 \end{eqnarray*}
\end{lem}
 \begin{proof}
  We take $\delta(t,z)=v_1(t,z)-v_2(t,z)$. Let us note that if  $(t,z)\in[0,h]\times\R$ then 
\small\begin{eqnarray*}
 \mathcal{L}\delta(t,z)&=&\int_{\R}K_2(y)g_2(v_2(t-h,z-ch-y))dy-\int_{\R}K_1(y)g_1(v_1(t-h,z-ch-y))dy.
  \end{eqnarray*}\normalsize
  If $g_2$  is a non decreasing function  by (\ref{vv}) we have 
 
 \small\begin{eqnarray*}
 \mathcal{L}\delta(t,z)&\geq&\int_{\R}K_2(y)g_2(v_1(t-h,z-ch-y))dy-\int_{\R}K_1(y)g_1(v_1(t-h,z-ch-y))dy.
  \end{eqnarray*}\normalsize
 But,  (\ref{vv}) and (\ref{v2}) imply $v_1(t-h,\cdot)\leq R$ so by (\ref{kk})  
  \begin{eqnarray*}
 \mathcal{L}\delta(t,z)&\geq& 0.
  \end{eqnarray*} 
 \noindent Analogously, if $g_1$  is a non-decreasing function,  we have

 \small\begin{eqnarray*}
 \mathcal{L}\delta(t,z)&\geq&\int_{\R}K_2(y)g_2(v_2(t-h,z-ch-y))dy-\int_{\R}K_1(y)g_1(v_2(t-h,z-ch-y))dy\geq 0.
  \end{eqnarray*}\normalsize
  
\noindent  Finally,  by using Proposition \ref{prop} with $d_3(t,x):=g_j(v_j(t-h,x-ch))/v_j(t-h,x-ch)$, $j=1,2$, we conclude that the function $\delta(t,z)$ is exponentially bounded on $[0,h]\times\R$. Then, since  $\delta(0,z)\leq 0$ for all $z\in\R$, the Phragm\`en-Lindel\"of principle   \cite{PW}[Chapter 3, Theorem 10] implies that $\delta(t,z)\leq 0$ for $(t,z)\in[0,h]\times\R$. The argument is repeated for intervals $[h,2h],[2h,3h]...\quad$
\end{proof}

\vspace{4mm}

\begin{lem}\label{sepa0}
Let $\pm c\geq\pm c_*^{\pm}$ and $\pm\lambda\geq\pm\lambda_1(c)$.  If $v(t,z)$ is a solution of (\ref{nle}) with initial datum $u_0(s,z)$ satisfying  the conditions of Theorem \ref{aes} and such that
\begin{eqnarray}
 u_{\infty}:=\sup_{(t,z)\in[-h,+\infty)\times\R}u(t,z)<\infty.
 \end{eqnarray}
 then there exit $\sigma'>0$ and $z_0'\in\R$ such that 
 \begin{eqnarray}
  u(t,\pm z)\geq\sigma'\quad\forall (t,z)\in[-h,\infty)\times[z_0',\infty).
    \end{eqnarray}
   \end{lem}
\begin{proof}
Fix $N>0$ and define $K^-(z):=\tau\chi_{[-N,N]} K(z)$ for some $\tau\geq 1$ such that $|K^-|_{L^1}=1$. Next, define the monotone function $g^-:\R_{\geq 0}\to\R_{\geq 0}$ by $g^-(u)= \tau^{-1}\min_{x\in[u,u_{\infty}]}g(x)$ and $g^-(u)=g(u_{\infty})$ for $u\geq u_{\infty}$. Clearly, $g^-$ holds {\bf (M)} with positive equilibrium $\kappa_-=\min_{x\in[\kappa,u_{\infty}]}g(x)$ and $L_{g^-}=\tau^{-1}g'(0)$. Also, $g^-(u)\leq g(u)$ for all $u\in[0, u_{\infty}]$. 


Next, by denoting $q_*:=\kappa_--\sigma>0$ without loss of generality, due to the asymptotic behavior of wavefronts in $-\infty$ (see \cite[Theorem 3 and Theorem 7]{AG})  there are a wavefront (monotone by Remark \ref{Mono}) $\phi^-_c$ to (\ref{nle}) (where $K$ and $g$ are replaced by $K^-$ and $g^-$, respectively)  and $z_{\phi^-_c}\in\R$    such that 

\begin{eqnarray*}\label{asymp}
\phi^-_c(z)-q_*\eta_{\lambda_1(c)}(z-b)\leq 0\leq u_0(s,\pm z)\quad \hbox{for all} \quad (s,z)\in[-h,0]\times(-\infty, z_0].
\end{eqnarray*}

\noindent where  $\pm b\geq \pm b_{0}^{\pm}$. By {\bf (IC)} we also have  

\begin{eqnarray*}\label{z0}
\phi_c^-(z)-q_*\eta_{\lambda_1(c)}(z-b)\leq u_0(s,\pm z)\quad \hbox{for all} \ (s,z)\in [-h,0]\times [z_0,+\infty)
\end{eqnarray*}
Thus,  for all $\pm c\geq \pm c^{\pm}_*$ we get 
\begin{eqnarray}\label{all1}
\phi_c^-(z)-q_*\eta_{\lambda_1(c)}(z-b)\leq u_0(s,\pm z)\quad \hbox{for all} \ (s,z)\in [-h,0]\times \R
\end{eqnarray}


Now, denote by $u^-(t,z)$ the solution to (\ref{nle}), with $g=g^-$, $K=K^-$ and $u^-(s,z)=u_0(s,z)$ for $(s,z)\in[-h,0]\times\R$. Then, because of Theorem \ref{Sttg}  for $\pm b\geq\pm b^{\pm}_0$ we have 
 \begin{eqnarray} \label{inq1}
 \phi^-_c( z)-q_*\eta(z-b)\leq u^-(t,z)\quad\forall (t,z)\in[-h,\infty)\times\R.
 \end{eqnarray}
  Thus,  there are $\sigma'>0$ and $ z_0' \in\R$ such that
 \begin{eqnarray}\label{sep}
  u^-(t,\pm z)\geq\sigma'\quad\forall (t,z)\in[-h,\infty)\times[z_0',\infty), \quad  \pm c\geq\pm c_*^{\pm}.
    \end{eqnarray}
    
  However, Lemma \ref{lem} (with $R=u_{\infty}$)  implies
    \begin{eqnarray}\label{inq2}
  u(t,\pm z)\geq u^-(t,\pm z)\quad\forall (t,z)\in[-h,\infty)\times\R, \quad  \pm c\geq\pm c_*^{\pm}.
    \end{eqnarray}
   
   Thus, the Lemma follows by (\ref{inq1}) and (\ref{inq2}).

    
  
\end{proof}
Now, we prove a key result in order to obtain our global stability result

 \begin{lem}\label{per}
Under the conditions of Theorem \ref{aes} for each $\epsilon\in(0,m_g)$ there exist $T_{\epsilon}=T_{\epsilon}(u_0)>0$ such that $\pm c\geq \pm c_*^{\pm}$ implies
 \begin{eqnarray}\label{inq0}
 m_g-\epsilon\leq u(t,\pm z)\leq M_g+\epsilon\quad\hbox{for all} \ (t,z)\in[T_{\epsilon},\infty)^2.
 \end{eqnarray}
  \end{lem}
 
 \begin{proof}
  We define $m_{\epsilon}:=\min_{x\in[\kappa, M_g+\epsilon/2g'(0)]}g(x)$,

  $$\bar{g}(u):=\max_{x\in[0,u]}g(x)\quad \hbox{and}\quad \underline{g}(u):=\min_{x\in[u,M_g+\epsilon/2g'(0)]}g(x).$$
  
  It is clear that these monotone functions (we define $\underline{g}(u)=g(M_g+\epsilon/2g'(0))$ for all $u\geq M_g+\epsilon/2g'(0)$) satisfy {\bf(M)} with positive equilibrium $ M_g$ and $m_{\epsilon}$, respectively. Also, $L_{\bar{g}}= L_{\underline{g}}=g'(0)$ and 
  \begin{eqnarray}\label{monog}
  g(u)\leq\bar{g}(u) \ \hbox{for all} \ u\geq 0\quad \hbox{and} \quad \underline{g}(u)\leq g(u) \ \hbox{for all} \ u\in[0, M_g+\epsilon/2g'(0)]. 
  \end{eqnarray}

 
 By denoting $\bar{u}(t)$ as the homogeneous solution to (\ref{nle}) with $g=\bar{g}$ and initial datum $\bar{u}(s)=|u_0|_{L^{\infty}_{h,\lambda}}$ for $s\in[-h,0]$, because of Lemma \ref{lem}(with $R=+\infty$)   we have
 \begin{eqnarray}\label{supes}
 u(t,z)\leq \bar{u}(t)\quad \forall (t,z)\in[-h,\infty)\times\R.
 \end{eqnarray} 
Thus, because of $M_g$ is the global attractor to $\bar{g}$ there is $t_{\epsilon}>0$ such that
\begin{eqnarray}\label{inq7}
u(t,z)\leq M_g+\epsilon/2g'(0) \quad \forall (t,z)\in[t_{\epsilon},\infty)\times\R,
\end{eqnarray}

Now, we procede to obtain the lower estimation. Denoting by $\underline{u}(t,z)$ the solution to (\ref{nle}) with $g=\underline{g}$ and initial datum $\underline{u}_0(s,z)=u(s+h+t_{\epsilon},z)$ 
for $(s,z)\in[-h,0]\times\R,$  then by (\ref{inq7}) and Lemma \ref{lem} (with $R=M_g+\epsilon/2g'(0)$)
  \begin{eqnarray}\label{sep0}
 u(t+t_{\epsilon}+h,z)\geq\underline{u}(t,z)\quad\forall (t,z)\in[0,\infty)\times\R.
     \end{eqnarray}

Next,  by (\ref{inq7}) we have $$u_{\infty}:=\sup_{(t,z)\in[-h,\infty)\times\R}u(t,z)<\infty,$$
so that Lemma \ref{sepa0} implies  there exit $\sigma'>0$ and $z_0'\in\R$ such that 
 \begin{eqnarray}\label{ult}
  \underline{u}(t,\pm z)\geq\sigma'\quad\hbox{for all} \ (t,z)\in[-h,+\infty)\times[z_0',\infty)\ \hbox{and} \ \pm c\geq \pm c^{\pm}_*.
    \end{eqnarray}






Then, we define 
  $$ 
  0<\alpha:=\frac{m_{\epsilon}-\epsilon/4}{\underline{g}(m_{\epsilon}-\epsilon/4)}<1,
  $$
 for $\epsilon$ enough small such that $g_{\alpha}:=\alpha\underline{g}$ satisfies {\bf(M)} with positive equilibrium $m_{\epsilon}-\epsilon/4$. Next, we set $\beta(t)$ as the solution to the problem
 \begin{eqnarray*}
 \beta'(t)&=&-\beta(t)+g_{\alpha}(\beta(t-h))\quad t>0\\
 \beta(s)&=& \sigma'\quad\quad s\in[-h,0],
 \end{eqnarray*}
 Without restriction: $\sigma'<m_{\epsilon}-\epsilon/4.$ So by \cite[Corollary 2.2, p. 82]{HS} $\beta(t)$ converges monotonetly to $m_{\epsilon}-\epsilon/4$.

  Then, by Theorem \ref{st} and Proposition \ref{pro}, for each $N>0$ there are $t_N>t_{\epsilon}$  and $z'\geq z'_0$ such that 
 \begin{eqnarray}\label{inq4}
 \underline{u}(t+t_N,\pm z)\geq m_{\epsilon}-\frac{\epsilon}{4}>\beta(t)\quad\forall (t,z)\in[0,\infty)\times[z',z'+N],\  \pm c\geq\pm c_*^{\pm}.
  \end{eqnarray} 
 
\noindent Now we consider $c\geq c^+_*$ and we fix $N$ large enough, such that
 \begin{eqnarray}\label{inq3}
 \alpha\leq \int_{-\infty}^{N-ch}K(y)dy.
 \end{eqnarray} 
  We define $\delta(t,z):=\beta(t)-\underline{u}(t+t_N+h,z)$. So, by (\ref{ult}) we obtain 
 \begin{eqnarray}\label{initi}
 \delta(s,z)\leq 0\quad\forall (s,z)\in[-h,0)\times[z',\infty),
  \end{eqnarray}
  and for $(t,z)\in[0,h]\times[z'+N,\infty)$ because of (\ref{initi}) and (\ref{inq3}) we have
 \begin{eqnarray*}
\mathcal{L}\delta(t,z)=\int_{\R}K(y)\underline{g}(\underline{u}(t+t_N,z-ch-y))dy-\alpha \underline{g}(\beta(t-h))\geq\\
\int_{-\infty}^{N-ch}K(y)[\underline{g}(\underline{u}(t+t_N,z-ch-y))-\underline{g}(\beta(t-h))]dy\geq 0.
 \end{eqnarray*}
 So by (\ref{inq4}) and Phragm\`en-Lindel\"of principle we conclude that 
 \begin{eqnarray*}
 \delta(t,z)\leq 0\quad \forall(t,z)\in[0,h]\times[z'+N,\infty),
  \end{eqnarray*}
 and by (\ref{inq4})
 \begin{eqnarray}\label{inq5}
 \delta(t,z)\leq 0\quad \forall(t,z)\in[0,h]\times[z',\infty).
  \end{eqnarray}
 Therefore using again (\ref{inq4}) and (\ref{inq5}) instead of (\ref{initi}) we can repeat the process for the intervals $[h,2h],[2h,3h]...$ to obtain
 \begin{eqnarray}\label{inq6}
 \beta(t)\leq \underline{u}(t+t_N+h,z)\quad \forall(t,z)\in[0,\infty)\times[z',\infty).
  \end{eqnarray}
 Finally, by (\ref{sep0}) and (\ref{inq6}) there exist $T_{\epsilon}(u_0)\geq t_N+t_{\epsilon}$
 $$
 m_g-g'(0)[\epsilon/2g'(0)]-\epsilon/2\leq m_{\epsilon}-\frac{\epsilon}{2}\leq u(t,z)\quad (t,z)\in[T_{\epsilon}(u_0),\infty)^2.
 $$
 
Otherwise, for  $c\leq c_*^-$ if we use the inequality (\ref{inq2}) and the same function $\beta(t)$ then the situation is completely analogous and therefore (\ref{inq0}) can be obtained. 
\end{proof}

  \subsection{Proof of Theorem \ref{aes}} 
  We will give the proof to the case $c\geq c_*^+$ since the proof for the case  $c\leq c_*^-$ is completely analogous.

  \begin{itemize}

  \item[(i)] We  take small $\epsilon_0>0$ such that  $\rho_{\epsilon_0}:=L_g([m_g-\epsilon_0,M_g+\epsilon_0])<1$ and   $\rho_{\epsilon_0}e^{\gamma_*h}< 1-\gamma_*$. 
  
  

 Note that by (\ref{rc1}) we get
\begin{eqnarray}\label{A1}
 e^{-\lambda z}|u(t,z)-\phi_c(z)|\leq \frac{|r_0|_{L^1_{h,\lambda}}}{A_h\sqrt{t}} e^{-\gamma_{\lambda} t}  \quad \ \forall \ t>h, z\in\R
 \end{eqnarray}

  Now, we consider a function $r:[-h,+\infty)\to\R_+$ given by $r(t):=q_0e^{-\gamma_*t}$ where $q_0\geq m_g+M_g$ will be fixed below. Then,  
for  $T_0:=\max\{T_{\epsilon_0}(u_0),T_{\epsilon_0}(\phi_c),h\}$ (according to Lemma \ref{per}) we define $\delta_{\pm}(t,z):=\pm[u(t+T_0+h,z)-\phi_c(z)]-r(t)$. So, by (\ref{inq7}) we obtain 
  $$\delta_{\pm}(s,z)\leq 0\ \hbox{for} \ (s,z)\in[-h,0]\times\R.$$   
  
 And, if $(t,z)\in[0,h]\times\R$  then by (\ref{A1}) and  Lemma \ref{per} we get  
 
  \begin{eqnarray}\label{L}
  \mathcal{L}\delta_{\pm}(t,z)=\pm\int_{\R}K(z-ch-y)[g(\phi_c(y))-g(u(t+T_0,y))]-\mathcal{L}r(t) 
  \end{eqnarray}
  \small\begin{eqnarray*}
&\geq&-[\frac{g'(0)|r_0|_{L^1_{h,\lambda}} e^{-\gamma_{\lambda}(t+T_0)}}{A_h\sqrt{t+T_0}}\int_{-\infty}^{T_0}e^{\lambda y}K(z-ch-y)dy+\rho_{\epsilon_0}\int_{T_0}^{+\infty}K(z-ch-y)r(t-h)dy]-\mathcal{L}r(t)\\
&\geq &- q_0e^{-\gamma_*t}[\frac{g'(0)|r_0|_{L^1_{h,\lambda}} e^{-\gamma_{\lambda} T_0}}{q_0A_h\sqrt{t+T_0}}\int_{-\infty}^{T_0}e^{\lambda y}K(z-ch-y)dy+ e^{\gamma_*h}\rho_{\epsilon_0}\int_{T_0}^{+\infty}K(z-ch-y)dy-1+\gamma_*]
\end{eqnarray*}\normalsize
Now, in the last inequality  since $\rho_{\epsilon_0}e^{\gamma_*h}< 1-\gamma_*$ we can choose $q_0$ large enough such that $\mathcal{L}\delta_{\pm}(t,z)\geq 0$ for all $(t,z)\in[0,h]\times\R$,  so that Phragm\`en-Lindel\"of principle implies  $\delta_{\pm}(t,z)\leq 0$ for $(t,z)\in[0,h]\times\R.$

Analogously, by using (\ref{A1})  and Lemma \ref{per} it is possible to repeat the process for the intervals $[h,2h],[2h,3h]...$ in order to obtain $\delta(t,z)\leq 0$ for all $(t,z)\in[-h,\infty)\times\R$.
   
   Finally,  as the $w(t,z)=u(t,z)-\phi_c(z)$ satisfies 
 $$  
 w_t(t,z)=w_{zz}(t,z)-cw_z(t,z)-w(t,z)+\int_{\R}K(z-ch-y)d_3(t,y)w(t-h,y)dy  
 $$  
where $d_3(t,y)=[g(u(t-h,y))-g(\phi_c(y))]/[u(t-h,y)-\phi_c(y)]$, by Proposition \ref{prop}, with $\lambda'=0$, we obtain  
  $$
  \sup_{(t,z)\in[-h,T_0]\times\R}|u(t,z)-\phi_c(z)|\leq qD^{[T_0/h]+1} 
  $$  
     
   By taking $C=\max\{q_0, qD^{[T_0/h]+1} e^{\gamma_* T_0} \}$ the result is followed.
   
    \item[(ii)] We take $\epsilon_0>0$ such that  $\rho_{\epsilon_0}:=L_g([m_g-\epsilon_0,M_g+\epsilon_0])<1$ and choose $d>h$ satisfying 
         \begin{eqnarray}\label{inq}
 \frac{\rho_{\epsilon_0}\sqrt{t+d}}{\sqrt{t+d-h}}+\frac{1}{2(t+d)}<1\quad t\geq-h.
  \end{eqnarray}    
    
 Now, we consider $r:[-h,+\infty)\to\R_+$ given by $r(t):=q_0/\sqrt{t+d}$, where $q_0\geq (m_g+M_g)/\sqrt{d}$  will be fixed below, and for  $T'_0=\max \{T_0,d\}$  ($T_0$ is taken like in part (i)) we define   $\delta_{\pm}(t,z):=\pm[u(t+T'_0+h,z)-\phi_c(z)]-r(t)$. So, by (\ref{inq0}) we have $$\delta_{\pm}(s,z)\leq 0\quad \hbox{for all} \ (s,z)\in[-h,0]\times\R.$$ And if $(t,z)\in[0,h]\times\R$, by (\ref{rc1}) and Lemma \ref{per} 
  
 $$
  \mathcal{L}\delta_{\pm}(t,z)=\pm\int_{\R}K(z-ch-y)[g(\phi_c(y))-g(u(t+T'_0,y))]-\mathcal{L}r(t) 
  $$
  $$
\geq-[\frac{g'(0)|r_0|_{L^1_{h,\lambda_c}}}{A_h\sqrt{t+T'_0}}\int_{-\infty}^{T'_0}e^{\lambda_c y}K(z-ch-y)dy+\rho_{\epsilon_0}\int_{T'_0}^{+\infty}K(z-ch-y)r(t-h)dy+\mathcal{L}r(t)]
$$
\begin{eqnarray*}
\geq-\frac{q_0}{\sqrt{t+d}}[\frac{g'(0)|r_0|_{L^1_{h,\lambda_c}}}{q_0 \ A_h}\frac{\sqrt{t+d}}{\sqrt{t+T'_0}}|K|_{L1_{\lambda_c}}+\rho_{\epsilon_0}\frac{\sqrt{t+d}}{\sqrt{t+d-h}}-1+\frac{1}{2(t+b)}]. 
\end{eqnarray*}
 However, by (\ref{inq}), in the last inequality we can choose $q_0$ large enough such that $\mathcal{L}\delta_{\pm}(t,z)\geq 0$ for all $(t,z)\in[0,h]\times\R$, so that Phragm\`en-Lindel\"of principle implies  $\delta_{\pm}(t,z)\leq 0$ for $(t,z)\in[0,h]\times\R.$ Repeating the process in the intervals $[h,2h],[2h,3h]...$ we obtain $\delta(t,z)\leq 0$ for all $(t,z)\in[-h,\infty)\times\R.$ The rest of the proof is  similar to part (i). 
 \end{itemize}


 
 
 
 
 \section*{Acknowledgments} This work was supported by FONDECYT (Chile) through the Postdoctoral  Fondecyt  2016 program with project number 3160473 and under the auspices of the Pontifical Catholic University of Chile.  

\end{document}